\documentclass[EJP]{ejpecp} % replace ECP by EJP if needed.
\usepackage{mathrsfs}
\usepackage{enumerate}

\DeclareMathOperator{\len}{len}
\DeclareMathOperator{\Leb}{Leb} \DeclareMathOperator{\interior}{int}

\SHORTTITLE{Limit of the Wulff crystal for isoperimetry in percolation}

\TITLE{Limit of the Wulff crystal when approaching criticality for isoperimetry in 2D percolation}

\AUTHORS{Chang-Long~Yao\footnote{Academy of Mathematics and Systems Science, Chinese Academy of Sciences, China.
    \EMAIL{deducemath@126.com}}}

\KEYWORDS{percolation; Cheeger constant; isoperimetry; shape theorem; Wulff crystal} % Separate items with ;

\AMSSUBJ{60K35; 82B43} % Edit. Separate items with ;
\SUBMITTED{March 8, 2023} % Edit.
\ACCEPTED{00, 0000} % Edit.

\VOLUME{0}
\YEAR{0000}
\PAPERNUM{0}
\DOI{10.1214/YY-TN}

\ABSTRACT{We consider isoperimetric sets, i.e., sets with minimal vertex boundary for a prescribed
volume, of the infinite cluster of supercritical site percolation on the triangular lattice.  Let $p$ be the percolation parameter and let $p_c$ be the critical point.  By adapting the proof of Biskup, Louidor, Procaccia and Rosenthal \cite{BLPR15} for isoperimetry in bond percolation on the square lattice, we show that the isoperimetric sets, when suitably rescaled, converge almost surely to a translation of the normalized Wulff crystal $\widehat{W}_p$.  More importantly, we prove that $\widehat{W}_p$ tends to a Euclidean disk as $p\downarrow p_c$.  This settles the site version of a conjecture proposed in \cite{BLPR15}.  A key input to the proof is the convergence of the limit shapes for near-critical Bernoulli first-passage percolation proved by the author recently.
}

\begin{document}

\section{Introduction}
\subsection{The model and main result}\label{s1}
The goal of an isoperimetric problem is to characterize sets of a prescribed volume with minimal boundary measure.
The classical isoperimetric problems are stated in the continuum and have a long history, but they have recently been studied in the discrete setting as well (see, e.g., \cite{BLPR15,CD20,Dem20a,Dem20b,GMPT17,Gol18a,Gol18b,Pet08} for such studies in bond percolation on $\mathbb{Z}^d$).  The Cheeger constant is a way of encoding such problems for graphs.  Given a graph $G(V(G), E(G))$, we define the \textbf{vertex boundary} $\partial_G A$ of a subset $A$ of $V(G)$ by
\begin{equation*}
\partial_G A:=\{v\in V(G)\backslash A: v\mbox{ is adjacent to a vertex in }A\},
\end{equation*}
where $v$ is adjacent to a vertex $u$ means that $(u,v)\in E(G)$.  For a subgraph $H\subset G$, we write $\partial_G H:=\partial_G V(H)$ for short.  For a finite set $A$, let $|A|$ denote its cardinality; for a finite graph $G$, write $|G|:=|V(G)|$.  The \textbf{Cheeger constant} of a finite graph $G$, also called the \textbf{isoperimetric constant}, is defined by
\begin{equation*}
\Phi_{G}:=\min\left\{\frac{|\partial_{G} H|}{|H|}:H\subset G,0<|H|\leq\frac{|G|}{2}\right\}.
\end{equation*}
The continuous version of this constant was introduced in the context of manifolds by Cheeger in his thesis \cite{Ch70}.
For additional background on isoperimetric problems, in particular, on the Cheeger constant defined by using the edge boundaries of graphs, see
Section 1 of \cite{BLPR15}.  In the present paper, we shall only consider the vertex-boundary version of the Cheeger constant, since our main result relies on the scaling limit of near-critical site percolation on the triangular lattice.

In this paper, we are interested in isoperimetric properties of random graphs arising from site percolation on the triangular lattice.  We refer the reader to the books \cite{BR06,Gri99} and the survey \cite{BD13} for background on percolation.  Let $\mathbb{T}$ denote the triangular lattice embedded in $\mathbb{C}$ (we identify the plane
$\mathbb{R}^2$ with the set $\mathbb{C}$ of complex numbers in the usual way), with vertex (i.e., site) set
\begin{equation*}
V(\mathbb{T}):=\{x+ye^{\pi i/3}\in\mathbb{C}:x,y\in\mathbb{Z}\},
\end{equation*}
and edge (i.e., bond) set $E(\mathbb{T})$ obtained by connecting all pairs
$u,v\in V(\mathbb{T})$ for which ${\|u-v\|_2=1}$, where $\|\cdot\|_2$ denotes the Euclidean norm.  Site percolation on $\mathbb{T}$ is defined as follows:  For each $p\in [0,1]$, we consider the i.i.d. family $\{\omega(v):v\in V(\mathbb{T})\}$ of Bernoulli random variables with parameter $p$, that is, $\omega(v)=1$ with probability $p$ and $\omega(v)=0$ with probability $1-p$.  This gives rise to a product probability measure $\mathbf{P}_p$ on the set
of configurations $\{0,1\}^{V(\mathbb{T})}$.  A vertex $v$ is declared \textbf{open} in the configuration $\omega$ if $\omega(v)=1$ and \textbf{closed} if $\omega(v)=0$.  We usually represent site percolation on $\mathbb{T}$ as a random two-coloring of the faces of the dual hexagonal lattice $\mathbb{H}$, each face centered at $v$ being \textbf{yellow} if $\omega(v)=1$ and \textbf{blue} if $\omega(v)=0$;  see Figure \ref{fppfig}.  Sometimes we shall view a vertex $v$ as the hexagon in $\mathbb{H}$ centered at $v$.

\begin{figure}
\centering
\includegraphics[height=0.35\textwidth]{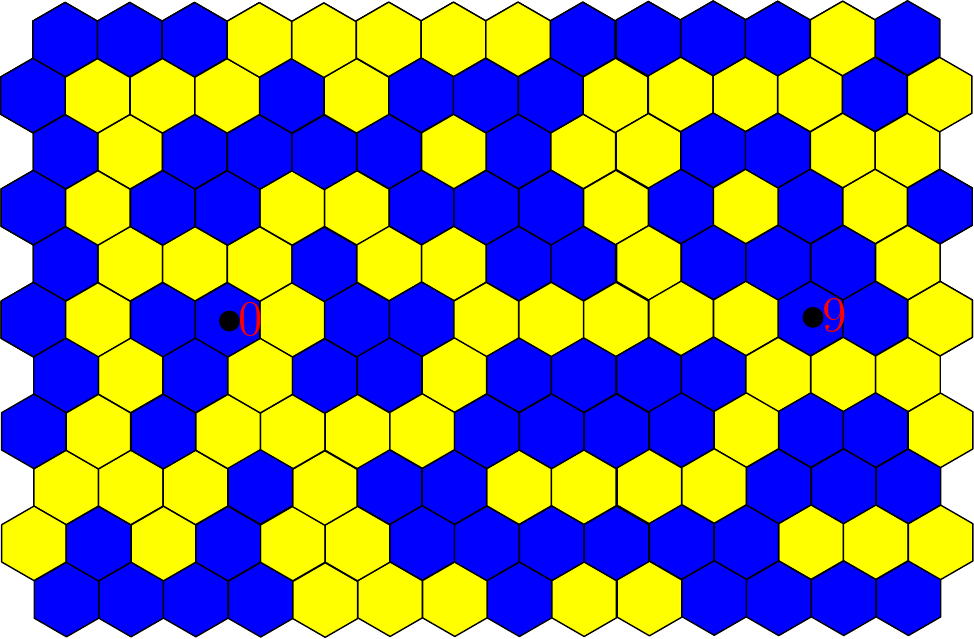}
\caption{A site percolation configuration on the triangular lattice $\mathbb{T}$.  Each
hexagon of the hexagonal lattice
$\mathbb{H}$ represents a vertex of $\mathbb{T}$, and is
colored yellow ($\omega(v)=1$) or blue ($\omega(v)=0$).  When this configuration is viewed as a Bernoulli first-passage percolation configuration, we have $T(0,9)=2$.}\label{fppfig}
\end{figure}

It is well known that site percolation on $\mathbb{T}$ exhibits a phase transition
at $p_c=1/2$: When $p\leq p_c$ there are almost surely no infinite open clusters, whereas when $p>p_c$ there is almost surely a unique infinite
open cluster.  We shall focus on the supercritical ($p>p_c$) regime, and denote by $\mathcal{C}_p^{\infty}$ the infinite (open) cluster.  Let ${\theta_p:=\mathbf{P}_p(0\in\mathcal{C}_p^{\infty})}$ be the density of $\mathcal{C}_p^{\infty}$ within $\mathbb{T}$.  Write $B_r:=[-r,r]^2\cap V(\mathbb{T})$ and  let $B_r(x):=x+B_r$ for $x\in\mathbb{R}^2$.  Let $\mathcal{C}_p^n$ denote the largest connected component of $\mathcal{C}_p^{\infty}\cap B_n$.

Following \cite{BLPR15}, we use two natural modified versions of the Cheeger constant to study the geometry of $\mathcal{C}_p^{\infty}$.  The first is
\begin{equation*}
\widetilde{\Phi}_{\mathcal{C}_p^n}:=\min\left\{\frac{\left|\partial_{\mathcal{C}_p^{\infty}} H\right|}{|H|}:H\subset \mathcal{C}_p^n,0<|H|\leq\frac{|\mathcal{C}_p^n|}{2}\right\}.
\end{equation*}
This modification is natural in the sense that, the vertex boundary of $H\subset \mathcal{C}_p^n$ is taken in the whole infinite cluster $\mathcal{C}_p^{\infty}$ instead of $\mathcal{C}_p^n$, which avoids giving advantage to subgraphs that touch the boundary of $B_n$.  Let $\mathscr{G}_{\mathcal{C}_p^n}$ be the set of minimizers of $\widetilde{\Phi}_{\mathcal{C}_p^n}$.

Another version of the Cheeger constant of $\mathcal{C}_p^{\infty}$, called the \textbf{anchored isoperimetric profile}, is defined by
\begin{equation*}
\Phi_{\mathcal{C}_p^{\infty},0}(n):=\min\left\{\frac{\left|\partial_{\mathcal{C}_p^{\infty}} H\right|}{|H|}:0\in H\subset \mathcal{C}_p^{\infty}, H\mbox{ connected }, 0<|H|\leq n\right\},
\end{equation*}
where we condition on the event $\{0\in\mathcal{C}_p^{\infty}\}$.  Let $\mathscr{G}_{\mathcal{C}_p^{\infty}}(n)$ be the set of minimizers of $\Phi_{\mathcal{C}_p^{\infty},0}(n)$.

To characterize the asymptotic shape of $\mathscr{G}_{\mathcal{C}_p^n}$ or $\mathscr{G}_{\mathcal{C}_p^{\infty}}(n)$, we need to introduce some notation.  For a continuous curve $\gamma: [0,1]\rightarrow \mathbb{R}^2$ and a norm $\rho$ on $\mathbb{R}^2$, let the $\rho$-length of $\gamma$ be defined as
\begin{equation}\label{e10}
\len_{\rho}(\gamma):=\sup_{N\geq 1}\sup_{0\leq t_0\leq\cdots\leq t_N\leq 1}\sum_{i=1}^{N}\rho(\gamma(t_i)-\gamma(t_{i-1})).
\end{equation}
A curve $\gamma$ is called rectifiable if $\len_{\rho}(\gamma)<\infty$ for any norm $\rho$ on $\mathbb{R}^2$.  If $\gamma$ is a Jordan curve, its interior $\interior(\gamma)$ is the unique bounded component of $\mathbb{R}^2\backslash\gamma$.

In Proposition \ref{p1}, we will define a norm $\beta_p$ associated with supercritical site
percolation on $\mathbb{T}$ with parameter $p$.  We define $\varphi_p$ as the solution of the following continuous isoperimetric problem:
\begin{equation}\label{e1}
\varphi_p:=\inf\{\len_{\beta_p}(\gamma):\gamma\mbox{ is a Jordan curve in }\mathbb{R}^2, \Leb(\interior(\gamma))=1\},
\end{equation}
where $\Leb$ stands for the Lebesgue measure on $\mathbb{R}^2$.  A minimizer of
$\varphi_p$ can be explicitly constructed, which is referred to as the Wulff
construction.  Define
\begin{equation}\label{e2}
W_p:=\bigcap_{u:\|u\|_2=1}\{x\in\mathbb{R}^2: u\cdot x\leq\beta_p(u)\}\quad\mbox{and}\quad\widehat{W}_p:=\frac{W_p}{\sqrt{\Leb(W_p)}},
\end{equation}
where $\cdot$ denotes the Euclidean scalar product.   So $\widehat{W}_p$ is $W_p$ normalized to have a unit area.  We call $W_p$ the \textbf{Wulff
crystal} associated to $\beta_p$.  Note that $W_p$ is a convex set and its boundary $\partial W_p$ is a simple curve.
Taylor \cite{Tay74,Tay75} showed that $\partial\widehat{W}_p$ is a minimizer of (\ref{e1}), and the minimizer is unique up to shifts.

For two compact sets $A,B$ in $\mathbb{R}^2$, the \textbf{Hausdorff distance} between $A$ and $B$ is defined by
\begin{equation*}
d_H(A,B):=\max\left\{\sup_{x\in A}\inf_{y\in B}\|x-y\|_{\infty}, \sup_{y\in B}\inf_{x\in A}\|x-y\|_{\infty}\right\},
\end{equation*}
where $\|\cdot\|_{\infty}$ denotes the $\ell^{\infty}$-norm on $\mathbb{R}^2$.

For $x\in\mathbb{R}^2$ and $r>0$, let $\mathbb{D}_r(x):=\{y\in\mathbb{R}^2:\|x-y\|_2\leq r\}$ denote the Euclidean disk centered at $x$ with radius $r$.  Write $\mathbb{D}_r:=\mathbb{D}_r(0)$.

Analogs of the following two theorems for the edge-boundary version of isoperimetry in bond percolation on $\mathbb{Z}^2$ were proved by Biskup, Louidor, Procaccia and Rosenthal \cite{BLPR15}.  We shall explain briefly in Section \ref{adapt} that the adaptation of the proofs in \cite{BLPR15} is straightforward to our setting once we define the norm $\beta_p$ by using the right-most paths for site percolation on $\mathbb{T}$.
\begin{theorem}[Cheeger constant, see Theorems 1.4 and 1.8 in \cite{BLPR15} for the analogs]\label{t4}
Suppose $p>p_c$.  Let $\varphi_p$ be the constant defined in (\ref{e1}).  Then, $\mathbf{P}_p$-almost surely,
\begin{equation*}
\lim_{n\rightarrow\infty}n\widetilde{\Phi}_{\mathcal{C}_p^n}=\frac{1}{\sqrt{2}}\theta_p^{-1}\varphi_p.
\end{equation*}
Moreover, for the minimizers the following holds: $\mathbf{P}_p$-almost surely,
\begin{equation*}
\max_{G\in\mathscr{G}_{\mathcal{C}_p^n}}\inf_{x\in\mathbb{R}^2}d_H(n^{-1}V(G),x+\sqrt{2}\widehat{W}_p)\rightarrow 0\quad\mbox{as }n\rightarrow\infty
\end{equation*}
and
\begin{equation*}
\max_{G\in\mathscr{G}_{\mathcal{C}_p^n}}\left|\frac{|G|}{\theta_p|B_n|/2}-1\right|\rightarrow 0\quad\mbox{as }n\rightarrow\infty.
\end{equation*}
\end{theorem}
\begin{theorem}[Isoperimetric profile, see Theorems 1.2 and 1.7 in \cite{BLPR15} for the analogs]\label{t3}
Suppose $p>p_c$.  Let $\varphi_p$ be the constant defined in (\ref{e1}).  Then, $\mathbf{P}_p(\cdot\mid 0\in \mathcal{C}_p^{\infty})$-almost surely,
\begin{equation*}
\lim_{n\rightarrow\infty}n^{1/2}\Phi_{\mathcal{C}_p^{\infty},0}(n)=\theta_p^{-1/2}\varphi_p.
\end{equation*}
Moreover, for the minimizers the following holds: $\mathbf{P}_p(\cdot\mid 0\in \mathcal{C}_p^{\infty})$-almost surely,
\begin{equation*}
\max_{G\in\mathscr{G}_{\mathcal{C}_p^{\infty}}(n)}\inf_{x\in\mathbb{R}^2}d_H(n^{-1/2}V(G),x+\theta_p^{-1/2}\widehat{W}_p)\rightarrow 0\quad\mbox{as }n\rightarrow\infty
\end{equation*}
and
\begin{equation*}
\max_{G\in\mathscr{G}_{\mathcal{C}_p^{\infty}}(n)}\big| |G|/n-1\big|\rightarrow 0\quad\mbox{as }n\rightarrow\infty.
\end{equation*}
\end{theorem}
In \cite{BLPR15}, after established the shape theorems for isoperimetry in bond percolation on $\mathbb{Z}^2$, the authors
conjectured that the normalized Wulff crystal tends to a Euclidean disk as $p\downarrow p_c$ (see Section 1.3 of \cite{BLPR15}).  Note that this conjecture was restated in Section 1.5 of \cite{Gol18b}.  Our main result, Theorem \ref{t1}, confirms the conjecture for the vertex-boundary version of isoperimetry in site percolation on $\mathbb{T}$.  To state this result, we need to recall a notion called \textbf{correlation length} (or characteristic length), which is a concept from near-critical percolation.  Roughly speaking, the system looks like critical percolation on scales smaller than correlation length, while notable super/sub-critical behavior emerges above this length.  There are several natural definitions of correlation length (see, e.g., Section 7.1 of \cite{Nol08}) which are of the same order of magnitude, that is, for any two of these lengths, although they tend to $+\infty$ when $p$ goes to $p_c$, the ratio between them is bounded away from $0$ and $+\infty$.

In the following, we shall define correlation length in terms of the box-crossing events.  Given a box $\Lambda=[x_1,x_2]\times[y_1,y_2]$ and a self-avoiding path $(v_0,v_1,\ldots,v_k)$ of $\mathbb{T}$, we call this path a \textbf{left-right (resp. top-bottom) crossing} of $\Lambda$ if $v_1,\ldots,v_{k-1}\in\Lambda$ and the line segments $\overline{v_0v_1}$ and $\overline{v_{k-1}v_k}$ intersect the left and right (resp. top and bottom) sides of $\Lambda$, respectively.  For each $\epsilon\in(0,1/2)$ and $p\in(p_c,1]$, let
\begin{equation*}
L_{\epsilon}(p):=\inf\left\{n\geq 1: \mathbf{P}_p\left[\exists\mbox{ a left-right closed crossing of $[0,n]^2$}\right]\leq\epsilon\right\}.
\end{equation*}
For any $\epsilon,\epsilon'\in(0,1/2)$, $L_{\epsilon}(p)$ has the same order as $L_{\epsilon'}(p)$; see, e.g., Corollary 37 in \cite{Nol08}.  So we shall take a fixed $\epsilon_0\in(0,1/2)$ and write $L(p):=L_{\epsilon_0}(p)$.  It was proved in \cite{SW01} that $L(p)=(p-p_c)^{-4/3+o(1)}$ as $p\downarrow p_c$.

\begin{theorem}[Main result]\label{t1}
Consider site percolation on the triangular lattice $\mathbb{T}$.  For the constant $\varphi_p$ defined in (\ref{e1}) (as in Theorems \ref{t4} and \ref{t3}), we have
\begin{equation}\label{e6}
\lim_{p\downarrow p_c}L(p)\varphi_p=2\sqrt{\pi}\nu,
\end{equation}
where $\nu$ is the limiting constant in Theorem \ref{t2} for Bernoulli first-passage percolation on $\mathbb{T}$.  Moreover, the normalized Wulff crystal $\widehat{W}_p$ converges in the Hausdorff metric to the Euclidean disk $\mathbb{D}_{1/\sqrt{\pi}}$ as $p\downarrow p_c$.
\end{theorem}
\emph{Idea of the proof of Theorem \ref{t1}.}  Garban, Pete and Schramm \cite{GPS18} established the existence and rotational invariance of the scaling limit of near-critical percolation on $\mathbb{T}$.  Based on this result, it was proved in \cite{Yao22} that the limit shape for Bernoulli FPP is asymptotically circular as $p\downarrow p_c$ (see Theorem \ref{t2} below), which is a key ingredient of the proof of Theorem \ref{t1}.  In the present paper, we show that the norm $\beta_p$ for isoperimetry is exactly the same as the time-constant norm $\mu_p$ for Bernoulli FPP (see Proposition \ref{p2}).  Then it is easy to obtain Theorem \ref{t1}.
\begin{remark}\label{r1}
Little is known about the geometry of $W_p$.  We believe that $W_p$ is not a Euclidean disk for any $p\in(p_c,1]$.  Garet, Marchand, Procaccia and Th\'{e}ret \cite{GMPT17} proved that $\varphi_p$ and $\widehat{W}_p$ are continuous in $p\in(p_c,1]$ for bond percolation on $\mathbb{Z}^2$, and their method applies to our setting as well.
\end{remark}
\subsection{Ingredients from Bernoulli first-passage percolation}\label{FPP}
First-passage percolation (FPP) was introduced
by Hammersley and Welsh in 1965 as a model for the spread of a fluid in a random medium; see \cite{ADH17} for a recent survey.  We shall focus on a special FPP, called \textbf{Bernoulli (site) FPP} on $\mathbb{T}$.  It is equivalent to site percolation on $\mathbb{T}$, but studied from the FPP point of view:
Given a site percolation configuration $\omega$ on $\mathbb{T}$, for each vertex $v$ we view $\omega(v)$ as the passage time of $v$.  \emph{Note that in the FPP literature, in particular in \cite{Yao22}, we have $\mathbf{P}_p(\omega(v)=0)=p=1-\mathbf{P}_p(\omega(v)=1)$, while in the present paper we keep $\mathbf{P}_p$ as in the Bernoulli percolation literature so that $\mathbf{P}_p(\omega(v)=1)=p=1-\mathbf{P}_p(\omega(v)=0)$.}

A \textbf{path} $\gamma$ in $\mathbb{T}$ from $x$ to $y$ of length $|\gamma|=n$ is a sequence $(v_0,\ldots,v_n)$ of vertices in $V(\mathbb{T})$ such that $v_i$ is adjacent to $v_{i-1}$ for all $i=1,\ldots,n$ and $x=v_0,y=v_n$.  Define the passage time of $\gamma$ by $T(\gamma):=\sum_{v\in \gamma}\omega(v)$.
For $x,y\in V(\mathbb{T})$, the \textbf{first-passage time} from $x$ to $y$ is
defined by
\begin{equation*}
T(x,y):=\inf\{T(\gamma):\gamma \mbox{ is a path from $x$ to $y$}\}.
\end{equation*}
See Figure \ref{fppfig} for an illustration.  For $x,y\in\mathbb{C}$, let $T(x,y):=T(\{x'\},\{y'\})$,
where $x'$ (resp. $y'$) is the vertex in $V(\mathbb{T})$ closest to
$x$ (resp. $y$).  Any possible ambiguity can be avoided by ordering
$V(\mathbb{T})$ and taking the vertex in $V(\mathbb{T})$ smallest for
this order.

It follows from the subadditive ergodic theorem that, for any
$z\in\mathbb{C}$, there is a constant $\mu_p(z)$ such that
\begin{equation}\label{e28}
\lim_{n\rightarrow\infty}\frac{T(0,nz)}{n}=\mu_p(z)\qquad\mbox{$\mathbf{P}_p$-a.s.
and in $L^1$}.
\end{equation}
We call $\mu_p(z)$ the \textbf{time constant}.  It is well known (see, e.g., Theorem 6.1 in \cite{Kes86}) that
\begin{equation}\label{e16}
\mu_p(1)=0\quad\mbox{if and only if}\quad p\leq p_c.
\end{equation}
Using (\ref{e28}) and (\ref{e16}), it is easy to deduce that
$\mu_p$ is a norm on $\mathbb{C}$ for each fixed $p>p_c$.

The fundamental object of study is the random set
\begin{equation*}
B(t):=\{z\in\mathbb{C}:T(0,z)\leq t\}.
\end{equation*}
The unit ball in the $\mu_p$-norm is called the \textbf{limit shape} and is denoted by
\begin{equation*}
\mathcal {B}_p:=\{z\in\mathbb{C}:\mu_p(z)\leq 1\}.
\end{equation*}
It is a convex, compact set with non-empty interior, and is the limit of $B(t)$ in the following sense (see, e.g., Theorem 2.17 in \cite{ADH17} for the Cox-Durrett shape theorem): For each $p\in(p_c,1]$ and each $\epsilon\in(0,1)$,
\begin{equation}
\mathbf{P}_p\left[(1-\epsilon)\mathcal {B}_p\subset
\frac{B(t)}{t}\subset(1+\epsilon)\mathcal {B}_p\mbox{ for all large
}t\right]=1.
\end{equation}

The following result is a crucial ingredient of the proof of Theorem \ref{t1}.
\begin{theorem}[Corollary 1.2 in \cite{Yao22}]\label{t2}
There exists a constant $\nu>0$, such that
\begin{equation*}
\lim_{p\downarrow p_c}L(p)\mu_p(z)=\nu\quad\mbox{uniformly in $z\in\mathbb{C}$ with $|z|=1$.}
\end{equation*}
In particular, when $p\downarrow p_c$, the normalized limit shape
$L(p)^{-1}\mathcal {B}_p$ converges in the Hausdorff metric to the Euclidean disk $\mathbb{D}_{1/\nu}$.
\end{theorem}
The following key observation enables us to apply Theorem \ref{t2} to our isoperimetric problem.
\begin{proposition}\label{p2}
Suppose $p>p_c$.  The time-constant norm $\mu_p$ is exactly the same as the boundary norm $\beta_p$ (defined in Proposition \ref{p1}).
\end{proposition}
\begin{remark}
It was proved in \cite{GMPT17} that $\varphi_p$ and $\widehat{W}_p$ (for the bond version on $\mathbb{Z}^2$) are continuous in $p\in(p_c,1]$, as mentioned in Remark \ref{r1}; the proof mainly relies on the continuity of $\beta_p$ in $p\in(p_c,1]$, given in Lemma 6.1 of \cite{GMPT17}.  By Proposition \ref{p2}, the continuity of $\beta_p$ follows immediately from the continuity of $\mu_p$, which is a special case of the continuity property of the time constant for general FPP proved by Cox and Kesten \cite{CK81}.
\end{remark}
\begin{remark}
In \cite{BLPR15}, the norm $\beta_p$ was constructed by using right-most paths, specific to dimension $d=2$.  In \cite{Dem20b,Gol18b}, to study isoperimetric problems for bond percolation in higher dimensions,  the authors used minimal cutsets in boxes or maximal flows through boxes to construct a suitable norm (corresponding to the surface tension in the percolation setting) on $\mathbb{R}^d, d\geq 2$.  It was noted in \cite{Gol18b} that when $d=2$, such cutsets are dual to paths and fall within the realm of FPP.  However, in these papers it was not stated that $\beta_p$ is in fact equal to $\mu_p$.
\end{remark}

\subsection{Related work and an open problem}
Besides the Wulff crystal $W_p$ and the limit shape $\mathcal {B}_p$, some other large random sets of site percolation on $\mathbb{T}$, when suitably normalized, also converge to a Euclidean disk as $p$ goes to $p_c$.  A related result and a conjecture are as follows.
\begin{itemize}
\item Duminil-Copin \cite{DC13} used the scaling limit of
near-critical percolation obtained in \cite{GPS18} to show that
the Wulff crystal (different from the one studied in the present paper) for subcritical percolation on $\mathbb{T}$
converges to a Euclidean disk as $p\uparrow p_c$.  Roughly
speaking, he proved that the typical shape of a cluster
conditioned to be large becomes round, when $p\uparrow p_c$.
\item Benjamini \cite{Ben10} made the following conjecture: Suppose $p>p_c$.  Condition that $0$ is in the infinite cluster. Consider the balls centered at 0 with radius $n$ for the graph distance on the infinite cluster.  Then the limiting shape of these balls becomes round as $p\downarrow p_c$.  This conjecture was restated by Duminil-Copin in \cite{DC13}.  It was noted below Conjecture 3.3 in \cite{Ben13} that the question seems hard, since metric properties do not follow from conformal geometry.
\end{itemize}

\subsection{Outline}
The rest of the paper is organized as follows.  In Section \ref{norm}, we introduce the notion of right-most paths, study geometric properties of such paths and define the boundary norm $\beta_p$.  In Section \ref{adapt}, we sketch the proofs of Theorems \ref{t4} and \ref{t3} by using the method from \cite{BLPR15}.  Section \ref{normequal} is devoted to the proof of the norm-equality observation, Proposition \ref{p2}.  In Section \ref{theorem}, we prove our main result, Theorem \ref{t1}.

\section{The boundary norm}\label{norm}
In this section we define the boundary norm $\beta_p$ that is used to construct the Wulff crystal $W_p$.  To define $\beta_p$, we require the notion of
right-most paths, which will be used to characterize the outer boundary of a finite subgraph of $\mathbb{T}$.  These two notions were first introduced in \cite{BLPR15} for bond percolation on $\mathbb{Z}^2$, and our definitions are adapted to site percolation on $\mathbb{T}$.  In Section \ref{def} we introduce the notions. In Section \ref{geo} we then give geometric properties of right-most paths.  Section \ref{normdefine} explains that Proposition \ref{p1}, concerning the definition of $\beta_p$, can be obtained by adapting the proofs of \cite{BLPR15} to our setting.

\subsection{Definitions of right-most paths and the boundary norm}\label{def}
Let $\overrightarrow{E}(\mathbb{T})$ be the oriented version of the edge set $E(\mathbb{T})$, where each edge $(u,v)$ in $E(\mathbb{T})$ is replaced by two oriented
edges $\langle u,v\rangle$ and $\langle v,u\rangle$.  For an oriented edge $\langle u,v\rangle$ from $u$ to $v$, we call $v$ (resp. $u$) the \textbf{head} (resp. \textbf{tail}) of $\langle u,v\rangle$.  A path in $\mathbb{T}$ is called \textbf{simple} if it traverses each edge in $\overrightarrow{E}(\mathbb{T})$ at most once.  Consider a path $\gamma=(v_0,\ldots,v_n)$ in $\mathbb{T}$.  If $v_0=v_n$, then the path is called a \textbf{circuit}; in this case we identify indices modulo $n$.  When $v_{i-1}$ and $v_{i+1}$ are well defined, the \textbf{right-boundary edges} at $v_i$ are obtained by listing all oriented
edges which start at $v_i$, beginning with but not including $\langle v_i,v_{i-1}\rangle$, proceeding in a
counterclockwise order and ending with but not including $\langle v_i,v_{i+1}\rangle$.  When $v_{i-1}$ or $v_{i+1}$ is not well defined, the set of right-boundary edges at $v_i$ is defined to be empty.  The \textbf{right-boundary vertices} of $v_i$ are the heads of the right-boundary edges at $v_i$.
The \textbf{right boundary} $\partial^+\gamma$ of $\gamma$ is the set of all right-boundary vertices of all vertices of $\gamma$.  Let $\partial^+\gamma(v_i)$ denote the set of right-boundary vertices of $v_i$ along $\gamma$.  The path $\gamma$ is called \textbf{right-most} if it is simple and does not use any vertex in $\partial^+\gamma$ and, moreover, for each $v_i$ with $v_{i-1}$ and $v_{i+1}$ well defined, $v_i$ has at least one right-boundary vertex.  See Figure \ref{rightmostpath} (b) for an illustration of the above notions; for a comparison, see also Figure \ref{rightmostpath} (a) that illustrates analogous notions in \cite{BLPR15} for bond percolation on $\mathbb{Z}^2$.

\begin{figure}
\centering
\includegraphics[height=0.35\textwidth]{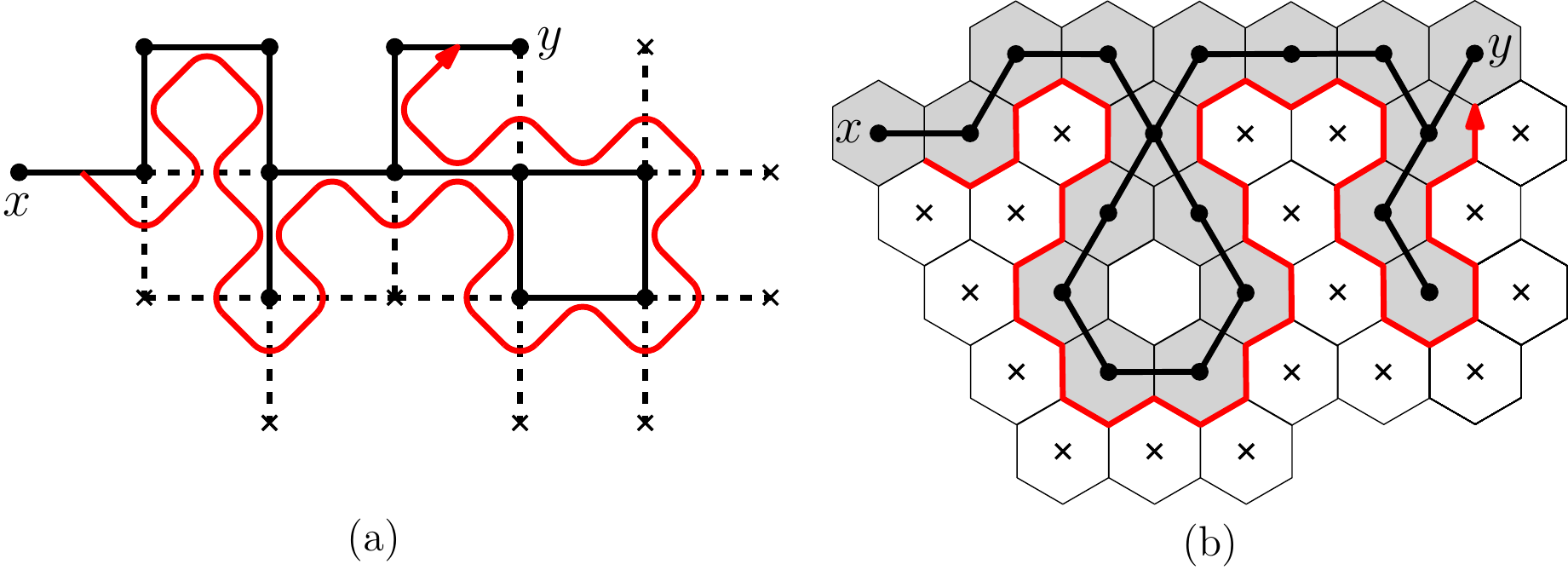}
\caption{(a): A right-most path (solid edges) from $x$ to $y$ on $\mathbb{Z}^2$.  The
dashed edges represent the right edge-boundary of this path.  The red curve in the medial graph represents the associated interface.  See
\cite{BLPR15} for more details.  (b): A right-most path $\gamma$ from $x$ to $y$ on $\mathbb{T}$.  The black bullets represent the vertices of $\gamma$, with the corresponding hexagons colored gray.  The black lines connecting the bullets represent the edges traversed by $\gamma$.  The small crosses represent the right boundary $\partial^+\gamma$ of $\gamma$.  The red curve is the associated interface $\partial^*\gamma$. }\label{rightmostpath}
\end{figure}

For $x,y\in V(\mathbb{T})$, let $\mathscr{R}(x,y)$ denote the set of all right-most paths from $x$ to $y$.  Given a right-most path $\gamma$, set
\begin{equation*}
\mathbf{b}(\gamma):=|\{v\in\partial^+\gamma:\mbox{ $v$ is open}\}|.
\end{equation*}
For two vertices $x,y$ that are in the same open cluster, we define the \textbf{right-boundary distance} from $x$ to $y$ by
\begin{equation*}
b(x,y):=\inf\{\mathbf{b}(\gamma):\gamma\in\mathscr{R}(x,y)\mbox{ and }\mbox{$\gamma$ is open}\}.
\end{equation*}
For each $x\in\mathbb{R}^2$, we let $\widetilde{x}=\widetilde{x}^{\mathcal{C}_p^{\infty}}$ denote the vertex of $\mathcal{C}_p^{\infty}$ that is nearest to $x$ in the $\ell^{\infty}$-norm, taking the smallest one in the lexicographic ordering of the differences between the vertices and $x$ in case there is a tie.  We now use right-most paths to define the norm $\beta_p$ on $\mathbb{R}^2$ (see Section \ref{normdefine} for a sketch of the proof):
\begin{proposition}[The boundary norm, see Theorem 2.1 and Proposition 2.2 in \cite{BLPR15} for the analogs]\label{p1}
Suppose $p>p_c$.  There exists a norm $\beta_p$ on $\mathbb{R}^2$ such that for any $x\in\mathbb{R}^2$,
\begin{equation*}
\beta_p(x):=\lim_{n\rightarrow\infty}\frac{b(\widetilde{0},\widetilde{nx})}{n}\quad\mbox{$\mathbf{P}_p$-a.s. and in $L^1(\mathbf{P}_p)$.}
\end{equation*}
Moreover, the convergence is uniform on
$\{x\in\mathbb{R}^2:\|x\|_2=1\}$.  Furthermore, the norm $\beta_p$ is invariant under symmetries of $\mathbb{T}$ that fix the origin.
\end{proposition}

\subsection{Geometry of right-most paths}\label{geo}
In this subsection we give some basic properties of right-most paths, which are analogs of the results in Section 2.2 of \cite{BLPR15}, with modifications adapted to our model.

To study the geometry of right-most paths on $\mathbb{T}$, we shall consider the dual of $\mathbb{T}$, the hexagonal lattice $\mathbb{H}$.  For each oriented edge $e$ of $\mathbb{T}$, we denote by $e^*$ its \textbf{dual edge} in $\mathbb{H}$, oriented so that the head of $e$ is on the right of $e^*$.  An \textbf{interface} is a sequence $(e_1^*,\ldots,e_n^*)$ of distinct oriented edges of $\mathbb{H}$ such that the head of $e_i^*$ is equal to the tail of $e_{i+1}^*$ for all $i=1,\ldots,n-1$ and, moreover, there exists a gray-white coloring of the hexagons of $\mathbb{H}$ so that for each edge in this sequence the hexagon on its left is gray and the hexagon on its right is white.  So an interface can be viewed as a simple curve in $\mathbb{R}^2$ and a piece of topological boundary of a cluster (considered as a union of hexagons) when we assign a suitable two-coloring of the faces of $\mathbb{H}$.  For an interface $\Gamma=(e_1^*,\ldots,e_n^*)$, let $\partial^+\Gamma$ denote the set of the heads of the primal edges $e_1,\ldots,e_n$.

Let $\gamma=(v_0,\ldots,v_n)$ be a right-most path with $n\geq 2$.  For each $1\leq i\leq n$, we list the right-boundary edges at $v_i$ in a counterclockwise order that begins with but not including $\langle v_i,v_{i-1}\rangle$, obtaining a sequence $(e_{j_{i-1}+1},\ldots,e_{j_i})$ with $j_0:=0$.  We associate to $\gamma$ the interface $\partial^*\gamma$ as follows (see Figure \ref{rightmostpath} (b)):
\begin{equation}\label{e19}
\partial^*\gamma:=\left\{
\begin{aligned}
&(e_{1}^*,\ldots,e_{j_1}^*,\ldots,e_{j_{n-2}+1}^*,\ldots,e_{j_{n-1}}^*)\quad\mbox{if $v_n\neq v_0$ (i.e., $\gamma$ is not a circuit),}\\
&(e_{1}^*,\ldots,e_{j_1}^*,\ldots,e_{j_{n-1}+1}^*,\ldots,e_{j_{n}}^*)\quad\mbox{otherwise (i.e., $\gamma$ is a circuit).}
\end{aligned}
\right.
\end{equation}
It is easy to check that $\partial^*\gamma$ is indeed an interface; see (\ref{item1}) of Proposition \ref{p4}.

A \textbf{cycle} is an interface with the head of its last edge equal to the tail of its first edge. (Note that this terminology differs from that of graph theory.)  For a cycle $\Gamma$ that surrounds at least two hexagons, a vertex $v$ in $\Gamma$ can be either of two types, according to whether the edge incident to $v$ that is not in (the unoriented version of) $\Gamma$ belongs to a hexagon surrounded by $\Gamma$ or not.  We call a vertex of the first (resp. second) type an \textbf{internal vertex} (resp. \textbf{external vertex}).  Note that an external vertex was called an e-vertex in \cite{CN06,CN07}, where ``e'' means ``external'' or ``exposed''.

\begin{figure}
\centering
\includegraphics[height=0.5\textwidth]{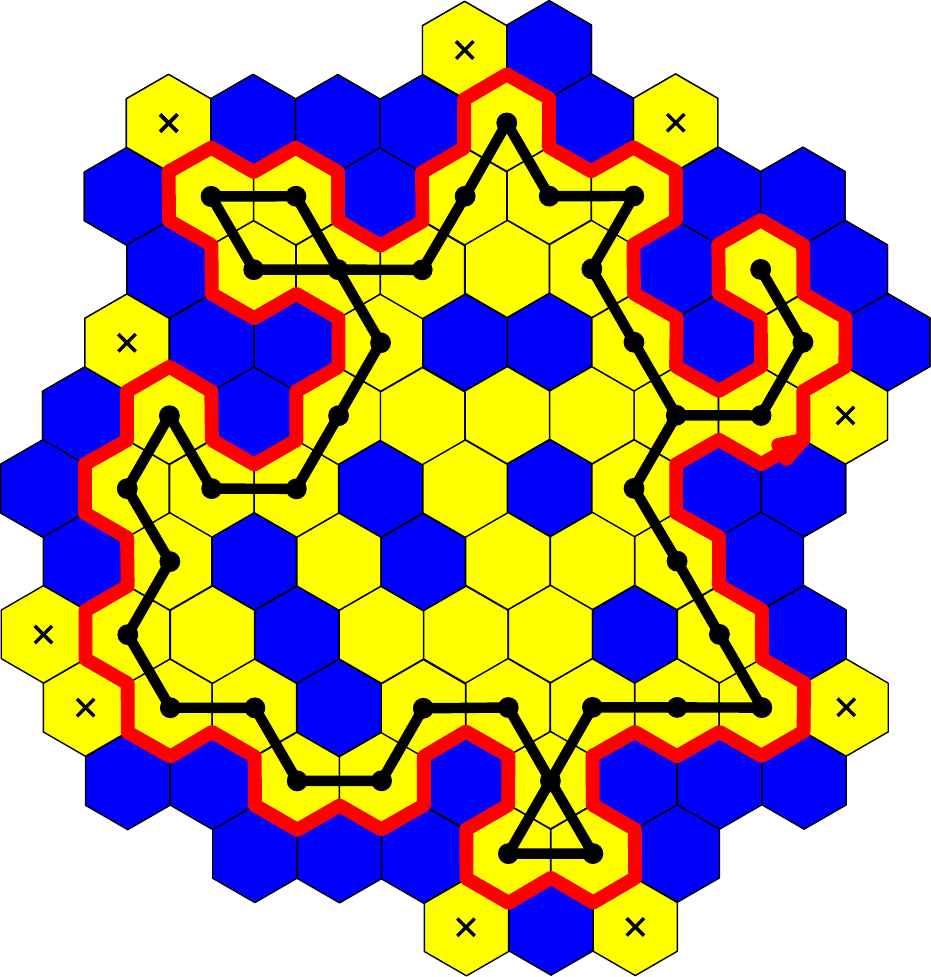}
\caption{A counterclockwise outer boundary interface $\Gamma$ (red curve) of a
finite connected subgraph $G$ of $\mathbb{T}$, with $V(G)$ represented by yellow hexagons inside $\Gamma$.  The black bullets represent the vertices of the inner vertex boundary $\partial^{i}G$ of $G$, which are also the vertices of the right-most circuit $\gamma$ satisfying $\partial^*\gamma=\Gamma$.  The black lines connecting the bullets represent the edges traversed by $\gamma$.  The small crosses represent the open vertices in $\partial^+\gamma$.  Here, $\mathbf{b}(\gamma)=10$.}\label{wulff}
\end{figure}

For a finite connected subgraph $G$ of $\mathbb{T}$, let $\mathcal{H}(G)$ denote the region that is the union of the hexagons of $\mathbb{H}$ centered at $V(G)$.  Then there is a cycle in $\mathbb{H}$ that goes around $\mathcal{H}(G)$ in the counterclockwise direction and coincides with the topological outer boundary of $\mathcal{H}(G)$; see Figure \ref{wulff}.  If the starting point of such a cycle is an internal vertex, then we call this cycle a \textbf{counterclockwise outer boundary interface} of $G$.  A \textbf{clockwise outer boundary interface} of $G$ is defined analogously, with its starting point being an external vertex.  The \textbf{outer vertex boundary} $\partial^{o}G$ and \textbf{inner vertex boundary} $\partial^{i}G$ of $G$ are defined respectively by
\begin{align*}
&\partial^{o}G:=\left\{v\in\partial_{\mathbb{T}} G:\begin{aligned}
&\mbox{there exists a path on $\mathbb{T}$ from $v$ to $\infty$
such that}\\
&\mbox{the only vertex of this path in $V(G)\cup\partial_{\mathbb{T}} G$ is $v$}
\end{aligned}
\right\}\mbox{ and}\\
&\partial^{i}G:=\{v\in V(G): v\mbox{ is ajacent to a vertex in }\partial^{o}G\}.
\end{align*}
The proposition below allows us to use right-most paths to study the ``shape'' of $G$, in particular, (\ref{item3}) says that there are two right-most circuits that lie respectively in $\partial^{i}G$ and $\partial^{o}G$, characterizing equally well the outer shape of $G$.
\begin{proposition}\label{p4}
Right-most paths and interfaces have the following properties:
\begin{enumerate}[(i)]
\item Let $\gamma$ be a right-most path with $|\gamma|\geq 2$.  Then the associated $\partial^*\gamma$ is an interface, and is a cycle when $\gamma$ is a circuit.  Moreover, $\partial^+\gamma=\partial^+(\partial^*\gamma)$.  Furthermore, for any right-most path $\gamma'\neq\gamma$ with $|\gamma'|\geq 2$, we have $\partial^*\gamma'\neq\partial^*\gamma$.\label{item1}
\item Let $\Gamma$ be an interface with the tail of its first edge and the head of its last edge not on the same hexagon in $\mathbb{H}$.  Then there is a unique right-most path $\gamma$ such that $\partial^*\gamma=\Gamma$.\label{item2}
\item Let $\Gamma$ (resp. $\Gamma'$) be a counterclockwise (resp. clockwise) outer boundary interface of a finite connected subgraph $G$ of $\mathbb{T}$ with $V(G)\geq 2$.  Then there is a unique right-most circuit $\gamma$ such that $\partial^*\gamma=\Gamma$.  Moreover, $\partial^+\gamma=\partial^o G$ and the set of vertices of $\gamma$ is equal to $\partial^i G$.  Similarly, there is a unique right-most circuit $\gamma'$ such that $\partial^*\gamma'=\Gamma'$.  Moreover, $\partial^+\gamma'=\partial^i G$ and the set of vertices of $\gamma'$ is equal to $\partial^o G$. \label{item3}
\end{enumerate}
\end{proposition}
\begin{proof}
(\ref{item1}) Let $\gamma=(v_0,\ldots,v_n)$.  First, we assume that $\gamma$ is not a circuit.  Since $\gamma$ is right-most, for all $i=1,\ldots,n-1$, $v_i$ has at least one right-boundary edge and the set $\{e_{j_{i-1}+1}^*,\ldots,e_{j_i}^*\}$ appearing in (\ref{e19}) is nonempty.  From the fact that $\gamma$ is a simple and does not use any vertex in $\partial^+\gamma$,  it is easy to see that $\partial^*\gamma=(e_{1}^*,\ldots,e_{j_{n-1}}^*)$ is a sequence of distinct oriented edges of $\mathbb{H}$ such that the head of $e_k^*$ is equal to the tail of $e_{k+1}^*$ for all $k=1,\ldots,j_{n-1}-1$ and, moreover, each hexagon centered at a vertex of $\gamma$ is on the left of one edge of $\partial^*\gamma$ but not on the right of any edge of $\partial^*\gamma$, which implies that there is a gray-white coloring of the hexagons of $\mathbb{H}$ so that for each edge in $\partial^*\gamma$ the hexagon on its left is gray and the hexagon on its right is white.  Thus, $\partial^*\gamma$ is indeed an interface.  Now assume that $\gamma$ is a circuit.  Similarly as above, one can prove that $\partial^*\gamma=(e_{1}^*,\ldots,e_{j_{n}}^*)$ is an interface.  Furthermore, it is easy to check that the head of $e_{j_{n}}^*$ is equal to the tail of $e_1^*$ since $v_n=v_0$, which means that $\partial^*\gamma$ is a cycle.  The fact that the sets $\partial^+\gamma$ and $\partial^+(\partial^*\gamma)$ are the same follows immediately from their definitions.

Let $\gamma'=(v_0',\ldots,v_m')$.  Similarly to (\ref{e19}), write $\partial^*\gamma'=(e_1'^*,\ldots,e_{j_{m-1}}'^*)$ if $v_m'\neq v_0'$, and $\partial^*\gamma'=(e_1'^*,\ldots,e_{j_m}'^*)$ otherwise.  Since $\gamma'\neq\gamma$, either $\gamma'=(v_0,\ldots,v_m)$ with $m<n$, or there is some $0\leq k\leq m$ such that $v_k'\neq v_k$ and $(v_0',\ldots,v_{k-1}')=(v_0,\ldots,v_{k-1})$.   In the former case, we have $e_{j_m}^*\neq e_{j_m}'^*$ (let $e_{j_m}'^*=\emptyset$ if $v_0'\neq v_m'$).  In the latter case, if $k=0$, then $e_1'^*\neq e_1^*$; otherwise, $e_{j_{k-1}}'^*\neq e_{j_{k-1}}^*$.  Therefore, $\partial^*\gamma'\neq\partial^*\gamma$.

(\ref{item2}) Let $\Gamma=(e_1^*,\ldots,e_m^*)$.  Note that the tail of $e_1^*$ is the common point of three hexagons of $\mathbb{H}$, in which we choose the hexagon that does not contain $e_1^*$ and denote it by $H_0$.  Let $H_1$ be the hexagon on the left of $e_1^*$.  If $H_1$ is also on the left of $e_2^*,\ldots,e_{j_1}^*$ but not on the left of $e_{j_1+1}^*$, then we let $H_2$ be the hexagon on the left of $e_{j_1+1}$.  We continue to construct the hexagons $H_3,\ldots,H_n$, similarly as above, with $H_n$ on the left of $e_{j_{n-1}+1}^*,\ldots,e_{j_n}^*=e_m^*$.  Denote by $H_{n+1}$ the hexagon pointed to by $e_m^*$.  Then we define $\gamma=(v_0,\ldots,v_{n+1})$, where $v_i$ is the vertex at the center of $H_i$ for $0\leq i\leq n+1$.  From the construction, it is clear that $\gamma$ is a path. The fact that the tail of $e_1^*$ and the head of $e_m^*$ are not on the same hexagon implies that $v_0\neq v_{n+1}$ and $n\geq 1$.  The orientation of $\Gamma$ and the fact that $\Gamma$ is self-avoiding ensure that $\gamma$ is simple.  By the construction, $H_1,\ldots,H_n$ are the hexagons on the left of $\Gamma$ with $H_0,H_{n+1}$ not on the left or right of $\Gamma$ and, moreover, the hexagons centered at the vertices of $\partial^+\gamma$ are on the right of $\Gamma$.  This implies that $\gamma$ does not use any vertex in $\partial^+\gamma$ since $\Gamma$ is an interface.  Furthermore, for each $1\leq i\leq n$, $v_i$ has at least one right-boundary vertex since the head of the primary edge $e_{j_i}$ is a right-boundary vertex of $v_i$.  Therefore, $\gamma$ is a right-most path.  From the construction, for each $1\leq i\leq n$, the primary edges $e_{j_{i-1}+1},\ldots,e_{j_i}$ are the right-boundary edges of $\gamma$ at $v_i$, which gives that $\partial^*\gamma=\Gamma$ since $v_0\neq v_{n+1}$.  The uniqueness follows from (\ref{item1}).

(\ref{item3}) We only give the proof of the statement for $\Gamma$, as $\Gamma'$ can be treated analogously.  Let $\Gamma=(e_1^*,\ldots,e_m^*)$.  Similarly to the proof of (\ref{item2}), we construct the hexagons $H_0,\ldots,H_n$.  Define the path $\gamma=(v_0,\ldots,v_n)$, where $v_i$ is the vertex at the center of $H_i$ for $0\leq i\leq n$.  The fact that $\Gamma$ is a cycle and the tail of $e_1^*$ is an internal vertex implies that $v_0=v_n$ and $n\geq 2$.  The fact that $\Gamma$ is counterclockwise and edge-simple ensures that $\gamma$ is simple.  By the construction, $H_0,\ldots,H_n$ are the hexagons on the left of $\Gamma$ and the hexagons centered at the vertices of $\partial^+\gamma$ are on the right of $\Gamma$.  This implies that $\gamma$ does not use any vertex in $\partial^+\gamma$ since $\Gamma$ is an interface.  Moreover, for each $1\leq i\leq n$, $v_i$ has at least one right-boundary vertex since the head of the primary edge $e_{j_i}$ is a right-boundary vertex of $v_i$.  Therefore, $\gamma$ is a right-most circuit.  From the construction, for each $1\leq i\leq n$, the primary edges $e_{j_{i-1}+1},\ldots,e_{j_i}$ are the right-boundary edges of $\gamma$ at $v_i$, which gives that $\partial^*\gamma=\Gamma$.  The uniqueness follows from (\ref{item1}).  Observe that $\partial^i G$ (resp. $\partial^o G$) is the set of vertices which are the centers of the hexagons on the left (resp. right) of $\Gamma$.  Then it follows from the above argument that $\partial^+\gamma=\partial^o G$ and the set of vertices of $\gamma$ is equal to $\partial^i G$.
\end{proof}

The next lemma enables us to estimate $|\partial^+\gamma|$ by the length $|\gamma|$ of a right-most path $\gamma$.
\begin{lemma}\label{l5}
For every right-most path $\gamma$,
\begin{equation*}
\frac{|\gamma|-1}{6}\leq|\partial^+\gamma|\leq 5|\gamma|.
\end{equation*}
\end{lemma}
\begin{proof}
Let $\gamma=(v_0,\ldots,v_n)$.  By the definition of right-most paths, for each $v_i$ with $v_{i-1}$ and $v_{i+1}$ well defined, the number of right-boundary vertices of $v_i$ is at least 1 and at most 5.  Moreover, it is easy to see that a vertex of $\mathbb{T}$ is a right-boundary vertex of at most 6 vertices in $\gamma$.  Then the lemma follows immediately.
\end{proof}

\begin{figure}
\centering
\includegraphics[height=0.12\textwidth]{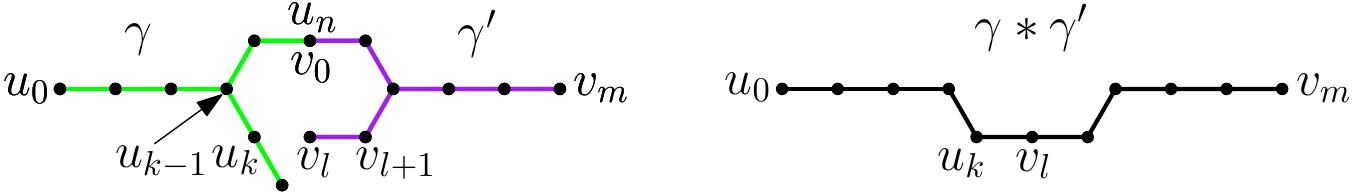}
\caption{The $*$-concatenation $\gamma*\gamma'$ of two right-most paths ${\gamma=(u_0,\ldots,u_n)}$ and ${\gamma'=(v_0=u_n,\ldots,v_m)}$.}\label{concatenation}
\end{figure}

Given two right-most paths $\gamma=(u_0,\ldots,u_n)$ and $\gamma'=(v_0,\ldots,v_m)$ with $v_0=u_n$, we now define a way to concatenate $\gamma$ and $\gamma'$, called the \textbf{$*$-concatenation}, such that the resulting path is a right-most path from $u_0$ to $v_m$.  If $u_0\in\gamma'$, then we set
$l:=\max\{i:v_i=u_0\}$ and let $\gamma*\gamma':=(v_l,\ldots,v_m)$; if $u_0\notin\gamma'$, then we set
\begin{equation*}
k:=\min\{i:u_i\mbox{ is adjacent to some vertex in }\gamma'\}\quad\mbox{and}\quad l:=\max\{i:v_i\mbox{ is ajacent to }u_k\},
\end{equation*}
and define $\gamma*\gamma':=(u_0,\ldots,u_k,v_l,\ldots,v_m)$.  See Figure \ref{concatenation} for
an example.  We now show that the $*$-concatenation does indeed yield a right-most path:
\begin{lemma}\label{l6}
For any $\gamma\in\mathscr{R}(x,y)$ and $\gamma'\in\mathscr{R}(y,z)$, we have $\gamma'':=\gamma*\gamma'\in\mathscr{R}(x,z)$. Moreover,
\begin{equation*}
|(\partial^+\gamma'')\backslash(\partial^+\gamma\cup\partial^+\gamma')|\leq 8.
\end{equation*}
\end{lemma}
\begin{proof}
Let $\gamma=(u_0=x,\ldots,u_n=y),\gamma'=(v_0=y,\ldots,v_m=z)$, and let $k,l$ be the associated quantities defined above Lemma \ref{l6}.  The case $u_0\in\gamma'$ is trivial, since in this case $\gamma''=(v_l=x,\ldots,v_m=z)$.  In the following, we assume that $u_0\notin\gamma'$.  Write
\begin{align*}
\gamma_L:=(u_0,\ldots,u_k),\quad\gamma_R:=(v_l,\ldots,v_m)\quad\mbox{and}\quad\gamma_M:=(u_{k-1},u_k,v_l,v_{l+1}),
\end{align*}
where we delete $u_{k-1}$ (resp. $v_{l+1}$) in $\gamma_M$ when $k=0$ (resp. $l=m$).

It is clear that $\gamma''$ is simple and $\gamma_L,\gamma_R$ are right-most paths.  Now we prove that $\gamma_M$ is also right-most.  Note that $\partial^+\gamma_M(u_k)\neq\emptyset$ when $k\neq 0$, since otherwise $u_{k-1}$ would be adjacent to $v_l$, contradicting the definition of $u_k$.  Moreover, $v_{l+1}$ is not in $\partial^+\gamma_M(u_k)$, as otherwise $u_k$ would be adjacent to $v_{l+1}$, contradicting the definition of $v_l$.  Similarly, $\partial^+\gamma_M(v_l)\neq\emptyset$ when $l\neq m$, and $u_{k-1}$ is not in $\partial^+\gamma_M(v_l)$.  Therefore, $\gamma_M$ is indeed a right-most path.
A similar argument shows that $\partial^+\gamma_L$ has no vertices in $\gamma_M\cup\gamma_R$, and $\partial^+\gamma_R$ has no vertices in $\gamma_L\cup\gamma_M$.

\begin{figure}
\centering
\includegraphics[height=0.2\textwidth]{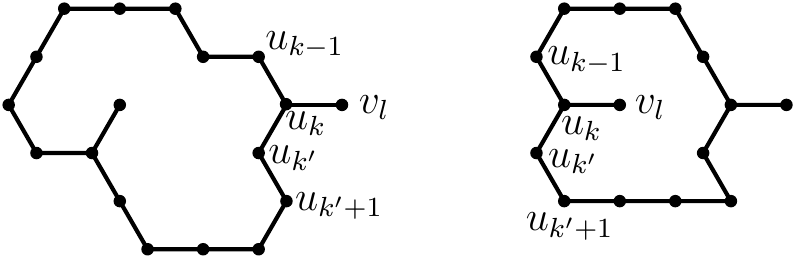}
\caption{\emph{Left}:  The circuit $(u_{k'},u_{k'+1},\ldots,u_k,u_{k'})$ is clockwise and $v_l$ lies in the exterior of this circuit.  \emph{Right}: The circuit $(u_{k'},u_{k'+1},\ldots,u_k,u_{k'})$ is counterclockwise and $v_l$ lies in the interior of this circuit.}\label{contradiction}
\end{figure}

Next, let us prove that $\partial^+\gamma_M$ has no vertices in $\gamma_L\cup\gamma_R$.  We first prove that $\partial^+\gamma_M$ has no vertices in $\gamma_L$.  Obviously, $\partial^+\gamma_M(v_l)$ has no vertices in $\gamma_L$, as otherwise there would exist some $0\leq k'<k$ such that $u_{k'}$ is adjacent to $v_l$, contradicting the definition of $u_k$.  Then, we need to show that $\partial^+\gamma_M(u_k)$ also has no vertices in $\gamma_L$; this is more involved than the $\partial^+\gamma_M(v_l)$ case.  Suppose for a contradiction that there exists some $0\leq k'<k-1$ such that $u_{k'}\in\partial^+\gamma_M(u_k)$.  It is easy to check that $u_{k'},u_{k-1}$ and $v_l$ are three distinct vertices adjacent to $u_k$ but not adjacent to each other, with $u_{k'}$ being the right-boundary vertex of $u_k$ along $(u_{k-1},u_k,v_l)$, and moreover, when $(u_{k'},u_{k'+1},\ldots,u_k,u_{k'})$ is a clockwise (resp. counterclockwise) circuit, then $v_l$ lies
in the exterior (resp. interior) of it; see Figure \ref{contradiction} for an illustration.  The definition of $u_k$ together with the facts that $u_0\notin\gamma'$ and $u_n=v_0$ implies that $k<n$.  Since $\gamma$ is right-most, it is easy to see that $u_{k+1}=u_{k'}$ and, moreover, when the circuit $(u_{k'},u_{k'+1}\ldots,u_k,u_{k+1}=u_{k'})$ is clockwise (resp. counterclockwise), $u_n$ lies on or in the interior (resp. exterior) of this circuit.  Thus, the path $(v_0=u_n,\ldots,v_l)$ must intersect the circuit $(u_{k'},u_{k'+1}\ldots,u_k,u_{k'})$, which implies that there is some $0\leq k''<k$ such that $u_{k''}$ is adjacent to some vertex in $\gamma'$, contradicting the definition of $u_k$.  Hence, $\partial^+\gamma_M(u_k)$ has no vertices in $\gamma_L$.  Then it follows that $\partial^+\gamma_M$ has no vertices in $\gamma_L$.  The proof of the fact that $\partial^+\gamma_M$ has no vertices in $\gamma_R$ is analogous to the $\gamma_L$ case.  Therefore, $\partial^+\gamma_M$ has no vertices in $\gamma_L\cup\gamma_R$.

The argument above implies that $\gamma''$ is a right-most path.  It is clear that ${\partial^+\gamma_L\subset\partial^+\gamma}$, ${\partial^+\gamma_R\subset\partial^+\gamma'}$ and ${\partial^+\gamma''=\partial^+\gamma_L\cup\partial^+\gamma_M\cup\partial^+\gamma_R}$.  Moreover, we have $\partial^+\gamma_M\leq 8$ since $u_{k-1}\neq v_l$ and $u_k\neq v_{l+1}$. Therefore, ${|(\partial^+\gamma'')\backslash(\partial^+\gamma\cup\partial^+\gamma')|}\leq {|\partial^+\gamma_M|}\leq 8$.
\end{proof}

\subsection{Proof of Proposition \ref{p1}}\label{normdefine}
Before giving the proof, we need a few basic facts about percolation.  The following lemma
says that the distance between any fixed vertex on $\mathbb{T}$ and the infinite cluster has exponential tails.
\begin{lemma}[see, e.g., Lemma 2.7 in \cite{BLPR15}]
For each $p>p_c$ there are $C_1,C_2>0$ such that for all $v\in V(\mathbb{T})$ and $r>0$,
\begin{equation}\label{e14}
\mathbf{P}_p(\|v-\widetilde{v}\|_{\infty}>r)\leq C_1\exp(-C_2r).
\end{equation}
\end{lemma}

For two vertices $u,v$ of $\mathbb{T}$ and a percolation configuration $\omega$ on $\mathbb{T}$, the \textbf{chemical distance} $D_{\omega}(u,v)$ is defined by
\begin{equation*}
D_{\omega}(u,v):=\inf\{\mbox{$|\gamma|$: $\gamma$ is an open path from $u$ to $v$}\}.
\end{equation*}
Denote by $u\leftrightarrow v$ the event that $u$ and $v$ are in the same open cluster.  Set $D_{\omega}(u,v)=\infty$ if $u\nleftrightarrow v$.  The estimate (\ref{e22}) below was proved by Antal and Pisztora \cite{AP96} (for supercritical bond percolation on $\mathbb{Z}^d$, but the proof adapts to our setting); as an extension of (\ref{e22}), the estimate (\ref{e23}) follows easily from (\ref{e22}) (see the proof of (2.16) in \cite{BLPR15}).
\begin{lemma}[Theorem 1.1 in \cite{AP96} and (2.16) in \cite{BLPR15}]\label{l7}
Suppose $p>p_c$.  There exist constants $C_1,C_2,C_3>0$ (depending on $p$) such that for all $v\in V(\mathbb{T})$,
\begin{equation}\label{e22}
\mathbf{P}_p(0\leftrightarrow v,D_{\omega}(0,v)>C_1\|v\|_{\infty})\leq C_2\exp(-C_3\|v\|_{\infty}).
\end{equation}
Moreover, there exist constants $C,C_4,C_5>0$ (depending on $p$), such that for all $v\in V(\mathbb{T})$ and $r\geq\|v\|_{\infty}$,
\begin{equation}\label{e23}
\mathbf{P}_p(0,v\in\mathcal{C}_p^{\infty},D_{\omega}(0,v)>Cr)\leq C_4\exp(-C_5r).
\end{equation}
\end{lemma}

\begin{proof}[Sketch of proof of Proposition \ref{p1}]
The proof is basically the same as that for Theorem 2.1 and Proposition 2.2 in \cite{BLPR15}, using the subadditive ergodic theorem, the geometric properties of right-most paths in Section \ref{geo} and some standard percolation inputs.  Therefore, we omit the details of the proof and just explain how to
address an additional issue here.

In fact, Theorem 2.1 of \cite{BLPR15} was stated in a slightly different form, and the analog of this form in our setting is the following: There exists a norm $\beta_p'$ on $\mathbb{R}^2$ such that for any $x\in\mathbb{R}^2$,
\begin{equation}\label{e20}
\beta_p'(x):=\lim_{n\rightarrow\infty}\frac{b([0],[nx])}{n}\quad\mbox{$\mathbb{P}_p$-a.s. and in $L^1(\mathbb{P}_p)$},
\end{equation}
where for each $x\in\mathbb{R}^2$, we let $[x]$ denote the vertex $v$ of $\mathcal{C}_p^{\infty}$ that is nearest to $x$ in the $\ell^{\infty}$-norm, taking the one with a minimal $\eta(v)$ in case there is a tie, and ${\{\eta(v):v\in V(\mathbb{T})\}}$ is a collection of i.i.d. random variables uniform on $[0,1]$ and independent of ${\{\omega(v):v\in V(\mathbb{T})\}}$ under $\mathbb{P}_p$, and the probability space is assumed to be large enough to support both $\omega$ and $\eta$, with $\mathbb{P}_p$ being the corresponding probability measure.  The advantage of the definition (\ref{e20}) is that the fact that $\beta_p'$ inherits all symmetries of the lattice $\mathbb{T}$ follows immediately from this definition (see Proposition 2.2 of \cite{BLPR15}).  However, extra randomness is introduced in the definition of $[x]$.  So, for simplicity we state Proposition \ref{p1} by using $\widetilde{x}$ instead of $[x]$, and explain below how to transfer the result for $[x]$ to that for $\widetilde{x}$.  Note that one can also work only with $\widetilde{x}$ without using $[x]$, by following the proofs in \cite{BLPR15} with some modifications.  Let us mention that Theorem 2.1 of \cite{BLPR15} was presented as Proposition 2.1 in \cite{GMPT17} in a form similar to Proposition \ref{p1} here.  Nevertheless, in the following we shall give an estimate (\ref{e21}) on $|b(\widetilde{0},\widetilde{nx})-b([0],[nx])|$, which implies that $\beta=\beta'$.  In fact, one can first use the arguments in \cite{BLPR15} to prove the statement for $\beta'$, and then obtain Proposition \ref{p1} by applying (\ref{e21}).

Let $x\in\mathbb{R}^2$.  Let $\gamma_1$ be an open path from $\widetilde{0}$ to $[0]$ such that $|\gamma_1|=D_{\omega}(\widetilde{0},[0])$, let $\gamma_2$ be an open right-most path from $[0]$ to $[nx]$ such that $\partial^+\gamma_2=b([0],[nx])$, and let $\gamma_3$ be an open path from $[nx]$ to $\widetilde{nx}$ such that $|\gamma_3|=D_{\omega}([nx],\widetilde{nx})$.  It is obvious that $\gamma_1$ and $\gamma_2$ are right-most paths.  Then we have
\begin{align*}
b(\widetilde{0},\widetilde{nx})&\leq\mathbf{b}(\gamma_1*\gamma_2*\gamma_3)\quad\mbox{by Lemma \ref{l6} and the definition of $b(\cdot,\cdot)$}\\
&\leq\mathbf{b}(\gamma_1)+\mathbf{b}(\gamma_2)+\mathbf{b}(\gamma_3)+16\quad\mbox{by Lemma \ref{l6}}\\
&\leq 5D_{\omega}(\widetilde{0},[0])+b([0],[nx])+5D_{\omega}([nx],\widetilde{nx})+16\quad\mbox{by Lemma \ref{l5}}.
\end{align*}
Similarly, we have $b([0],[nx])\leq 5D_{\omega}([0],\widetilde{0})+b(\widetilde{0},\widetilde{nx})+5D_{\omega}(\widetilde{nx},[nx])+16$.  Thus,
\begin{equation}\label{e18}
|b(\widetilde{0},\widetilde{nx})-b([0],[nx])|\leq 5D_{\omega}(\widetilde{0},[0])+5D_{\omega}([nx],\widetilde{nx})+16.
\end{equation}
Therefore, there exist constants $C_1,C_2,C_3>0$ (depending on $p$) such that for all $n\geq 1$ and all $x\in\mathbb{R}^2$,
\begin{align}
&\mathbf{P}_p(|b(\widetilde{0},\widetilde{nx})-b([0],[nx])|\geq\sqrt{n})\nonumber\\
&\leq\mathbf{P}_p(5D_{\omega}(\widetilde{0},[0])+5D_{\omega}([nx],\widetilde{nx})+16\geq\sqrt{n})\quad\mbox{by (\ref{e18})}\nonumber\\
&\leq\mathbf{P}_p(\|\widetilde{0}\|_{\infty}>C_1\sqrt{n})+\mathbf{P}_p(\|nx-\widetilde{nx}\|_{\infty}>C_1\sqrt{n})\nonumber\\
&\quad+\sum_{u,v\in\mathcal{C}_{\infty}\cap B_{C_1\sqrt{n}}}\mathbf{P}_p\left(5D_{\omega}(u,v)+8\geq\frac{\sqrt{n}}{2}\right)+\sum_{u,v\in\mathcal{C}_{\infty}\cap B_{C_1\sqrt{n}}(nx)}\mathbf{P}_p\left(5D_{\omega}(u,v)+8\geq\frac{\sqrt{n}}{2}\right)\nonumber\\
&\leq C_2\exp(-C_3\sqrt{n})\quad\mbox{by (\ref{e14}) and (\ref{e23})}.\label{e21}
\end{align}
\end{proof}

\section{Proofs of Theorems \ref{t4} and \ref{t3}}\label{adapt}
\begin{proof}[Sketch of proofs of Theorems \ref{t4} and \ref{t3}]
The proofs are essentially the same as the proofs of Theorems 1.2, 1.4, 1.6, 1.7 and 1.8 in \cite{BLPR15}, relying on Proposition \ref{p1},
the properties of right-most paths in Section \ref{geo} and some standard percolation inputs.  Hence, we shall only outline the strategy of the proofs in the following and refer the reader to \cite{BLPR15} for the details.

\textbf{Step 1.} (Concentration estimates. See Section 3 of \cite{BLPR15}.) We need to prove two key concentration estimates: One is for $b(\widetilde{0},\widetilde{x})$.  The idea is to apply Kesten's Azuma-type concentration inequality for martingales with bounded increments.  However, the increments of the martingale for $b(\widetilde{0},\widetilde{x})-\mathbf{E}_p b(\widetilde{0},\widetilde{x})$ are not bounded.  Instead, we apply Kesten's inequality to a modified right-boundary distance $\widehat{b}(0,x)$ which allows to use not fully open paths at a huge penalty.  Moreover, with high probability the quantities $\widehat{b}(0,x)$ and $b(\widetilde{0},\widetilde{x})$ are close to each other when $\|x\|_2$ is large.  Then we obtain the concentration estimate for $b(\widetilde{0},\widetilde{x})$ from that for $\widehat{b}(0,x)$.  The second key estimate is for a geometric concentration on the straight line segment joining $0$ to $x$ of the right-most paths nearly minimizing $b(\widetilde{0},\widetilde{x})$, and its proof is based on the first concentration estimate and Lemma \ref{l6}, using the $*$-concatenation to construct a nearly optimal right-most open path.

\textbf{Step 2.} (Approximating circuits by simple closed curves and vice versa.  See Section 4 of \cite{BLPR15}.)  This step involves two approximations: The ``circuits-to-curves'' approximation says that when $R\rightarrow\infty$, with high probability (quantitative version) for any ``long'' and open right-most circuit $\gamma$ in $B_R$ with the region surrounded by $\gamma$ not too small, there is a rectifiable simple closed curve $\lambda$ such that $\lambda$ approximates $\gamma$ ``well'' and ${\mathbf{b}(\gamma)\geq (1-\epsilon)\len_{\beta_p}(\lambda)}$.  The ``curves-to-circuits'' approximation says that, for any rectifiable simple closed curve $\lambda$ with $\interior(\lambda)$ being convex, as the scaling factor $R\rightarrow\infty$, with high probability there is an open right-most circuit $\gamma$ such that $\gamma$ approximates $R\lambda$ ``well'' and ${\mathbf{b}(\gamma)\leq(1+\epsilon)\len_{\beta_p}(R\lambda)}$.  The proofs of these results use polygonal approximations, the concentration estimates obtained in Step 1 and the properties of right-most paths in Section \ref{geo}.

\textbf{Step 3.} (Final proof. See Section 5 of \cite{BLPR15}.) Finally, one can prove Theorems \ref{t4} and \ref{t3} by using the approximation estimates given in Step 2 and some basic percolation results, such as the fact that, when $R\rightarrow\infty$, with high probability for any ``long'' right-most circuit $\gamma$ in $B_R$ with the region surrounded by $\gamma$ not too small, the ratio of the number of vertices of $\mathcal{C}_p^{\infty}$ surrounded by $\gamma$ to the number of vertices of $\mathbb{T}$ surrounded by $\gamma$ is approximately equal to $\theta_p$.
\end{proof}

\section{Proof of Proposition \ref{p2}}\label{normequal}
First, let us prove the following result:
\begin{lemma}\label{l1}
Suppose $p>p_c$.  For any $x\in\mathbb{R}^2$, we have $\beta_p(x)\geq\mu_p(x)$.
\end{lemma}
\begin{proof}
For simplicity, we prove the lemma in the case $x=1$; the proof extends immediately to the general case.  Suppose $\widetilde{0}\neq\widetilde{n}$.  Let $\gamma$ be an open right-most path from $\widetilde{0}$ to $\widetilde{n}$ with $\mathbf{b}(\gamma)=b(\widetilde{0},\widetilde{n})$.  Recall that $\partial^*\gamma$ denotes the interface associated with $\gamma$.  We call a circuit $(v_0,\ldots,v_m)$ a Jordan circuit if $(v_0,\ldots,v_{m-1})$ is a self-avoiding path and $m\geq 2$.  Note that the edges traversed by a Jordan circuit form a Jordan curve. If $\mathcal{C}_1,\ldots,\mathcal{C}_k$ is a sequence of disjoint open Jordan circuits separating $\widetilde{0}$ from $\widetilde{n}$, then $\partial^*\gamma$ must intersect each of the Jordan curves of these circuits, and for each oriented edge in $\partial^*\gamma$ that intersects one such Jordan curve, the open hexagon on the right of the edge belongs to the corresponding Jordan circuit and is in $\partial^+\gamma$.  Thus we have $\mathbf{b}(\gamma)\geq k$.  For $u,v\in V(\mathbb{T})$, let $N(u,v)$ denote the maximal number of disjoint open Jordan circuits separating $u$ from $v$.  Then the above observation implies that
\begin{equation}\label{e24}
b(\widetilde{0},\widetilde{n})\geq N(\widetilde{0},\widetilde{n}).
\end{equation}
By Proposition 2.5 (ii) of \cite{Yao22}, for any $u,v\in V(\mathbb{T})$, conditioned on $\omega(u)=\omega(v)=0$, we have $N(u,v)=T(u,v)$ almost surely.  Thus, for any $u,v\in V(\mathbb{T})$, almost surely
\begin{equation}\label{e25}
N(u,v)\geq T(u,v)-2.
\end{equation}
Combining (\ref{e24}) and (\ref{e25}), we get that almost surely
\begin{equation}\label{e17}
b(\widetilde{0},\widetilde{n})\geq T(\widetilde{0},\widetilde{n})-2\geq T(0,n)-T(0,\widetilde{0})-T(n,\widetilde{n})-2.
\end{equation}
Dividing both sides of (\ref{e17}) by $n$, letting $n\rightarrow\infty$ and applying Proposition \ref{p1}, (\ref{e28}) and (\ref{e14}), we get $\beta_p(1)\geq\mu_p(1)$.
\end{proof}
To show that $\beta_p=\mu_p$, it remains to show the inequality in the other direction:
\begin{lemma}\label{l2}
Suppose $p>p_c$.  For any $x\in\mathbb{R}^2$, we have $\beta_p(x)\leq\mu_p(x)$.
\end{lemma}
Before proving this lemma, we collect some relevant notation and percolation preliminaries.  For $w,h>0$, define the box $\Lambda_{w,h}:=[0,w]\times[-h,h]$.
For $\theta\in[0,2\pi]$, we set $\Lambda_{w,h}^{\theta}:=e^{i\theta}\cdot\Lambda_{w,h}$ and write $\Lambda_{w,h}^{\theta}(x):=x+\Lambda_{w,h}^{\theta}$ for $x\in \mathbb{R}^2$.  For a tilted box of the form $\Lambda_{w,h}^{\theta}$, we call the segment
$e^{i\theta}\cdot\{\mbox{the left side of }\Lambda_{w,h}\}$ (resp. $e^{i\theta}\cdot\{\mbox{the right side of }\Lambda_{w,h}\}$) the left (resp. right) side of $\Lambda_{w,h}^{\theta}$.  Then a left-right crossing of $\Lambda_{w,h}^{\theta}$ can be defined similarly as that of the box $[x_1,x_2]\times[y_1,y_2]$ (see Section \ref{s1}).

For $x\in\mathbb{R}^2\backslash\{0\}$ and $h>0$, define the \textbf{box-crossing time} (also called cylinder passage time) of the box $\Lambda_{\|x\|_2,h}^{\arg(x)}$ by
\begin{equation*}
T(0,x;h):=\inf\{T(\gamma):\mbox{$\gamma$ is a left-right crossing of $\Lambda_{\|x\|_2,h}^{\arg(x)}$}\}.
\end{equation*}
The following lemma says that $T(0,nx;h(n))$ has the same strong law of large numbers as $T(0,nx)$ if $h(n)\rightarrow\infty$ and $h(n)/n\rightarrow 0$ when $n\rightarrow\infty$.
\begin{lemma}\label{l3}
Suppose $p>p_c$.  For every $x\in\mathbb{R}^2\backslash\{0\}$ and every height function ${h:\mathbb{N}\rightarrow \mathbb{R}^+}$ satisfying $h(n)\rightarrow\infty$ and $h(n)/n\rightarrow 0$ as $n\rightarrow\infty$, we have
\begin{equation*}
\lim_{n\rightarrow\infty}\frac{T(0,nx;h(n))}{n}=\mu_p(x)\qquad\mbox{$\mathbf{P}_p$-a.s}.
\end{equation*}
\end{lemma}
\begin{proof}
Rossignol and Th\'{e}ret \cite{RT10} proved that for general FPP on $\mathbb{Z}^d$, when $n$ goes to infinity, the maximal flow (passage times of the edges are viewed as capacities) between the top and the bottom of the box $\Lambda_{\|nx\|_2,h(n)}^{\arg(x)}$ satisfies a strong law of large numbers, see Theorem 2.3 in \cite{RT10b} and the paragraph just below this theorem which states that the limiting constant for the maximal flow is equal to the corresponding time constant for FPP. (For the case that $h(n)$ has the same order as $n$ we refer the interested reader to Theorem 2.8 and Corollary 2.10 in \cite{RT10b}.)  The proof for the site version on $\mathbb{T}$ is essentially the same.  In the Bernoulli case the maximal flow between the top and the bottom of $\Lambda_{\|nx\|_2,h(n)}^{\arg(x)}$ is just the maximal number of disjoint open top-bottom crossings of this box, denoted by $\tau(0,nx;h(n))$.  Using Proposition 2.5 (i) in \cite{Yao22}, we observe that for all $n,h(n)\geq 2$,
\begin{equation*}
\tau(0,nx;h(n))\leq T(0,nx;h(n))\leq\tau(0,nx;h(n)-1)+2.
\end{equation*}
From this and the strong law of large numbers for $\tau(0,nx;h(n))$, Lemma \ref{l3} follows immediately.
\end{proof}
We shall use the next lemma to construct an open right-most path from $\widetilde{0}$ to $\widetilde{ne^{i\theta}}$ in a thin cylinder.
\begin{lemma}\label{l4}
For each $p>p_c$, there are $C_1,C_2>0$ such that for all $n\geq 10,\theta\in[0,2\pi]$ and $x\in\mathbb{R}^2$,
\begin{equation*}
\mathbf{P}_p\left(
\begin{aligned}
&\mbox{there exist open left-right crossings of $\Lambda_{n,\sqrt{n}/2}^{\theta}(x)$,}\\
&\mbox{and all such crossings lie in $\mathcal{C}_p^{\infty}$}
\end{aligned}
\right)\geq 1-C_1\exp(-C_2\sqrt{n}).
\end{equation*}
\end{lemma}
\begin{proof}
It is well known that (see, e.g., Theorem 5.4 in \cite{Gri99}) for $p>p_c$ there exists $C_3>0$ such that for all $n\geq 1$,
\begin{equation}\label{e13}
\mathbf{P}_p(\mbox{there is a closed path from 0 to some vertex outside $\mathbb{D}_n$})\leq\exp(-C_3n).
\end{equation}
Suppose $n\geq 10$.  Observe that if for each closed vertex $v$ within Euclidean distance 1 from the bottom side of $\Lambda_{n,\sqrt{n}/2}^{\theta}(x)$, the closed cluster containing $v$ is contained in $\mathbb{D}_{\sqrt{n}-2}(v)$, then there is an open left-right crossing of $\Lambda_{n,\sqrt{n}/2}^{\theta}(x)$.  It follows from this observation and (\ref{e13}) that for each $p>p_c$ there are $C_4,C_5>0$ such that for all $n\geq 10,\theta\in[0,2\pi]$ and $x\in\mathbb{R}^2$,
\begin{equation}\label{e15}
\mathbf{P}_p(\mbox{there exist open left-right crossings of $\Lambda_{n,\sqrt{n}/2}^{\theta}(x)$})\geq 1-C_4\exp(-C_5\sqrt{n}).
\end{equation}

It is well known that (see, e.g., (8.20) in \cite{Gri99}) for $p>p_c$ there exist $C_6,C_7>0$ such that for all $n\geq 1$,
\begin{equation*}
\mathbf{P}_p(\mbox{there is an open path from 0 to a vertex outside $\mathbb{D}_n$ and $0\notin\mathcal{C}_p^{\infty}$})\leq C_6n^2\exp(-C_7n).
\end{equation*}
Combining this with (\ref{e15}), it is easy to get Lemma \ref{l4}.
\end{proof}

\begin{proof}[Proof of Lemma \ref{l2}]
For simplicity, we shall prove the lemma in the case $x=1$; the proof extends easily to the general case.  Fix any $\epsilon>0$.  For our purpose, we need to
show that for each large $n$, with high probability we can construct an open right-most path $\gamma$ from $\widetilde{0}$ to $\widetilde{n}$ such that $\mathbf{b}(\gamma)\leq(1+\epsilon)\mu_p(1)n$.  First, define the events
\begin{align*}
&\mathcal{A}_n:=\left\{
\begin{aligned}
&\mbox{there exist open left-right crossings of $[0,n]\times[2\sqrt{n},3\sqrt{n}]$,}\\
&\mbox{open top-bottom crossings of $[0,\sqrt{n}]\times[-\sqrt{n},3\sqrt{n}]$ and}\\
&\mbox{open top-bottom crossings of $[n-\sqrt{n},n]\times[-\sqrt{n},3\sqrt{n}]$,}\\
&\mbox{and all such crossings lie in $\mathcal{C}_p^{\infty}$}
\end{aligned}
\right\},\\
&\mathcal{E}_n(\epsilon):=\{T(0,n;\sqrt{n})\leq(1+\epsilon/2)\mu_p(1)n\},\\
&\mathcal{F}_n:=\left\{
\begin{aligned}
&\mbox{$D_{\omega}(x,y)\leq 2C\sqrt{n}$ for all $x,y\in B_{\sqrt{n}}(0)\cap\mathcal{C}_p^{\infty}$,}\\
&\mbox{and $D_{\omega}(x,y)\leq 2C\sqrt{n}$ for all $x,y\in B_{\sqrt{n}}(n)\cap\mathcal{C}_p^{\infty}$}
\end{aligned}
\right\},
\end{align*}
where $C$ is the constant in (\ref{e23}).  By Lemmas \ref{l3} and \ref{l4} and (\ref{e23}), for all $n$ large enough (depending on $\epsilon$),
\begin{equation}\label{e5}
\mathbf{P}_p(\mathcal{A}_n\cap\mathcal{E}_n(\epsilon)\cap\mathcal{F}_n)\geq 1-\epsilon.
\end{equation}

\begin{figure}
\centering
\includegraphics[height=0.37\textwidth]{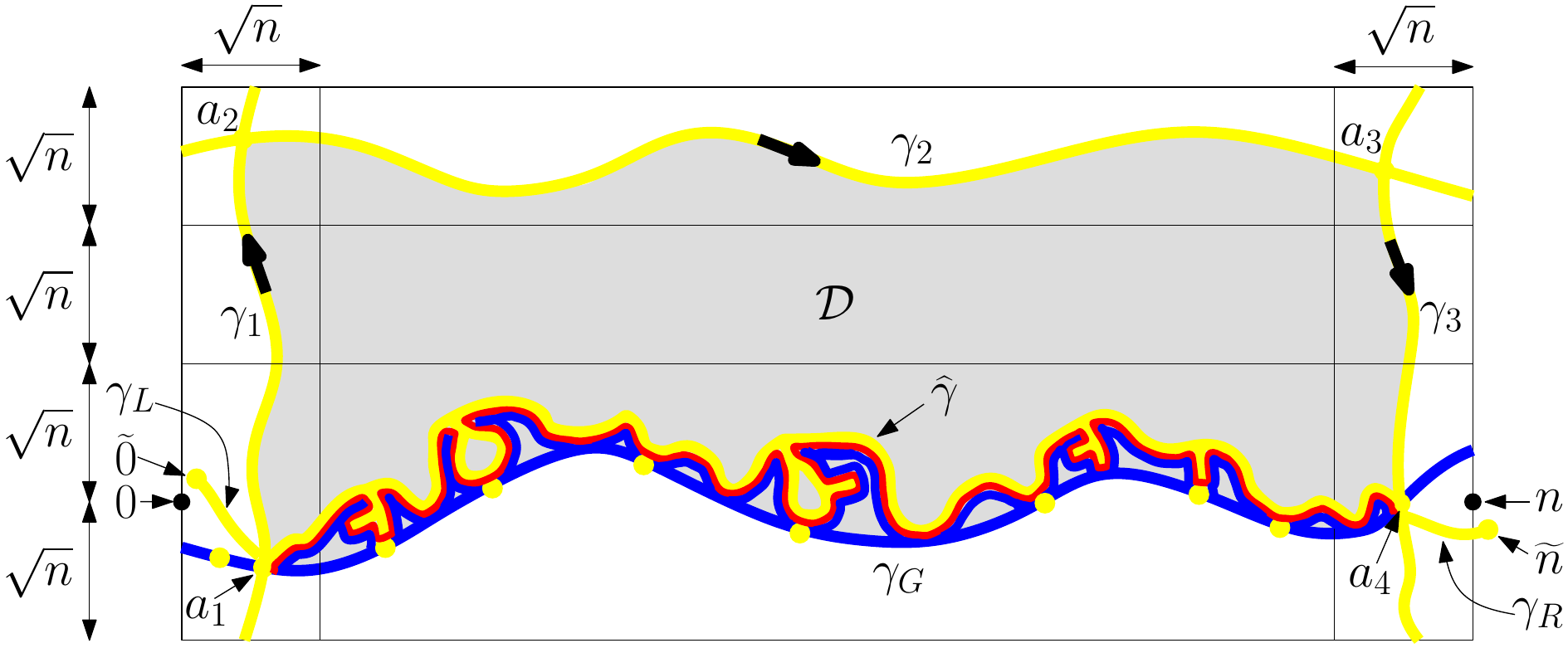}
\caption{Construction of the right-most path $\gamma=\gamma_L*\widehat{\gamma}*\gamma_R$ from $\widetilde{0}$ to $\widetilde{n}$.}\label{key}
\end{figure}
In the following, we assume that $n$ is sufficiently large and the event $\mathcal{A}_n\cap\mathcal{E}_n(\epsilon)\cap\mathcal{F}_n$ occurs.  The construction of the path $\gamma$ is illustrated in Figure \ref{key}.  We denote by $\gamma_1$ an open bottom-top crossing of $[0,\sqrt{n}]\times[-\sqrt{n},3\sqrt{n}]$, $\gamma_2$ an open left-right crossing of $[0,n]\times[2\sqrt{n},3\sqrt{n}]$,  and $\gamma_3$ an open top-bottom crossing of $[n-\sqrt{n},n]\times[-\sqrt{n},3\sqrt{n}]$, respectively.  Let $\gamma_G$ (here ``$G$'' refers to geodesic) denote a left-right crossing of $[0,n]\times[-\sqrt{n},\sqrt{n}]$ with $T(\gamma_G)=T(0,n;\sqrt{n})$.
For a fixed $\omega$, we can choose these four crossings in some deterministic way.  The paths $\gamma_1,\gamma_2$ and $\gamma_3$ are oriented as shown in Figure \ref{key}.  When traveling along $\gamma_1$, we let $a_1$ (resp. $a_2$) denote the last (resp. first) intersection vertex of $\gamma_1$ with $\gamma_G$ (resp. $\gamma_2$); when traveling along $\gamma_3$, we let $a_3$ (resp. $a_4$) denote the last (resp. first) intersection vertex of $\gamma_3$ with $\gamma_2$ (resp. $\gamma_G$).  The usual concatenation of the piece of $\gamma_1$ from $a_1$ to $a_2$, the piece of $\gamma_2$ from $a_2$ to $a_3$ and the piece of $\gamma_3$ from $a_3$ to $a_4$ yields a self-avoiding open path from $a_1$ to $a_4$, denoted by $\gamma'$.  Let $\gamma''$ be the subpath of $\gamma_G$ from $a_1$ to $a_4$.  Note that $\gamma'$ intersects $\gamma''$ only at $a_1$ and $a_4$.  Then the concatenation of $\gamma''$ and the reversal of $\gamma'$ forms a circuit.  Denote by $\mathcal{D}$ the discrete region (i.e., a union of hexagons) that is composed of this circuit and the region surrounded by it.  Now we define a new configuration $\omega'$ from the original configuration $\omega$ by redeclaring the open vertices in $\{\gamma''\setminus\{a_1,a_4\}\}$ to be closed while leaving the states of all other vertices unchanged.  For $\omega'$ restricted to $\mathcal{D}$, the interface between the open cluster containing $\gamma'$ and the closed cluster containing $\{\gamma''\setminus\{a_1,a_4\}\}$ induces an open right-most path $\widehat{\gamma}$ from $a_1$ to $a_4$ in $\mathcal{D}$, with all the hexagons of $\widehat{\gamma}$ adjacent to the interface (see (\ref{item2}) of Proposition \ref{p4}).  Observe that $\widehat{\gamma}$ is also open in $\omega$ and, by construction, all the vertices of $\partial^+\widehat{\gamma}$ are closed in $\omega'$ and all the open vertices of $\partial^+\widehat{\gamma}$ in $\omega$ are contained in $\{\gamma''\setminus\{a_1,a_4\}\}$.  In the rest of the proof we focus on $\omega$.  The above argument implies that
\begin{equation}\label{e3}
\mathbf{b}(\widehat{\gamma})\leq T(\gamma'')-2\leq T(\gamma_G)\leq (1+\epsilon/2)\mu_p(1)n,
\end{equation}
By the construction, it is clear that $a_1\in\mathcal{C}_p^{\infty}\cap B_{\sqrt{n}}(0)$ and $a_4\in\mathcal{C}_p^{\infty}\cap B_{\sqrt{n}}(n)$, which implies that  $\widetilde{0}\in B_{\sqrt{n}}(0)$ and  $\widetilde{n}\in B_{\sqrt{n}}(n)$.  Let $\gamma_L$ be a shortest open path from $\widetilde{0}$ to $a_1$, and let $\gamma_R$ be a shortest open path from $a_4$ to $\widetilde{n}$, where ``shortest'' means in the lattice distance.  Note that $\gamma_L,\gamma_R$ are right-most paths.  By successively $*$-concatenating $\gamma_L,\widehat{\gamma}$ and $\gamma_R$, we obtain the open right-most path $\gamma$ from $\widetilde{0}$ to $\widetilde{n}$, such that
\begin{align}
\mathbf{b}(\gamma)&\leq \mathbf{b}(\widehat{\gamma})+\mathbf{b}(\gamma_L)+\mathbf{b}(\gamma_R)+16\quad\mbox{by Lemma \ref{l6}}\nonumber\\
&\leq (1+\epsilon/2)\mu_p(1)n+5D_{\omega}(\widetilde{0},a_1)+5D_{\omega}(a_4,\widetilde{n})+16\quad\mbox{by (\ref{e3}) and Lemma \ref{l5}}\nonumber\\
&\leq (1+\epsilon/2)\mu_p(1)n+20C\sqrt{n}+16\quad\mbox{since $\mathcal{F}_n$ occurs}.\label{e4}
\end{align}
Letting $\epsilon\downarrow 0$, then (\ref{e5}), (\ref{e4}) together with Proposition \ref{p1} implies that $\beta_p(1)\leq\mu_p(1)$.
\end{proof}
\begin{proof}[Proof of Proposition \ref{p2}]
Proposition \ref{p2} follows from Lemmas \ref{l1} and \ref{l2} immediately.
\end{proof}

\section{Proof of Theorem \ref{t1}}\label{theorem}
\begin{proof}[Proof of Theorem \ref{t1}]
First, let us prove (\ref{e6}).  Write $\mathbb{U}:=\{x\in\mathbb{R}^2:\|x\|_2=1\}$.  Theorem \ref{t2} and Proposition \ref{p2} imply that for each $\epsilon\in(0,\nu)$, there exists $p_0\in(p_c,1)$ such that for all $p\in(p_c,p_0)$ and all $u\in\mathbb{U}$,
\begin{equation}\label{e8}
|L(p)\beta_p(u)-\nu|\leq\epsilon.
\end{equation}
It follows from (\ref{e10}) that for any norm $\rho$ on $\mathbb{R}^2$ and any rectifiable Jordan curve ${\gamma:[0,1]\rightarrow\mathbb{R}^2}$ with $\Leb(\interior(\gamma))=1$, we have
\begin{equation}\label{e7}
\len_{\rho}(\gamma)=\sup_{N\geq 1}\sup_{0\leq t_0\leq\cdots\leq t_N\leq 1}\sum_{i=1}^{N}\rho\left(\frac{\gamma(t_i)-\gamma(t_{i-1})}{\|\gamma(t_i)-\gamma(t_{i-1})\|_2}\right)\|\gamma(t_i)-\gamma(t_{i-1})\|_2.
\end{equation}
Combining (\ref{e7}) and (\ref{e8}), we obtain that for each $p\in(p_c,p_0)$,
\begin{equation}\label{e9}
|L(p)\len_{\beta_p}(\gamma)-\nu\len_{\|\cdot\|_2}(\gamma)|\leq\epsilon\len_{\|\cdot\|_2}(\gamma).
\end{equation}
Let $\gamma_p$ denote a Jordan curve such that $\len_{\beta_p}(\gamma_p)=\varphi_p$ and $\Leb(\interior(\gamma_p))=1$ (i.e., $\gamma_p$ is a shift of $\partial\widehat{W}_p$), and let $\mathbf{C}$ denote a Euclidean circle of radius $1/\sqrt{\pi}$ (so ${\Leb(\interior(\mathbf{C})=1}$).
Then (\ref{e9}) implies that for all $p\in(p_c,p_0)$,
\begin{align*}
&L(p)\len_{\beta_p}(\gamma_p)\leq L(p)\len_{\beta_p}(\mathbf{C})\leq(\nu+\epsilon)\len_{\|\cdot\|_2}(\mathbf{C})\quad\mbox{and}\\
&\nu\len_{\|\cdot\|_2}(\mathbf{C})\leq\nu\len_{\|\cdot\|_2}(\gamma_p)\leq\left(\frac{\nu}{\nu-\epsilon}\right)L(p)\len_{\beta_p}(\gamma_p).
\end{align*}
Then, letting $\epsilon\downarrow 0$, we get
\begin{equation*}
\lim_{p\downarrow p_c}L(p)\varphi_p=\nu\len_{\|\cdot\|_2}(\mathbf{C})=2\sqrt{\pi}\nu,
\end{equation*}
which concludes the proof of (\ref{e6}).

Next, we prove the convergence of $\widehat{W}_p$ as $p\downarrow p_c$.  For any $x\in W_p$, it follows from the definition of $W_p$ that
$u\cdot x\leq\beta_p(u)$ for all $u\in\mathbb{U}$, where $\cdot$ denotes the Euclidean scalar product.  Then, by (\ref{e8}), for each $p\in(p_c,p_0)$ we have
\begin{equation*}
L(p)u\cdot x\leq L(p)\beta_p(u)\leq\nu+\epsilon,
\end{equation*}
which gives that $L(p)x\in\mathbb{D}_{\nu+\epsilon}$.  Therefore, for each $p\in(p_c,p_0)$,
\begin{equation}\label{e11}
L(p)W_p\subset\mathbb{D}_{\nu+\epsilon}.
\end{equation}
By (\ref{e8}), for all $p\in(p_c,p_0)$, $x\in\mathbb{D}$ and $u\in\mathbb{U}$, we have
\begin{equation*}
\nu u\cdot x\leq\nu\leq\left(\frac{\nu}{\nu-\epsilon}\right)L(p)\beta_p(u),
\end{equation*}
which gives that $\nu x\in\left(\frac{\nu}{\nu-\epsilon}\right)L(p)W_p$. Thus, for each $p\in(p_c,p_0)$,
\begin{equation}\label{e12}
\mathbb{D}_{\nu}\subset\left(\frac{\nu}{\nu-\epsilon}\right)L(p)W_p.
\end{equation}
Combining (\ref{e11}) and (\ref{e12}) and letting $\epsilon\downarrow 0$, we obtain that $d_H(L(p)W_p,\mathbb{D}_{\nu})\rightarrow 0$ as $p\downarrow p_c$.  Thus, ${d_H(\widehat{W}_p,\mathbb{D}_{1/\sqrt{\pi}})\rightarrow 0}$ as $p\downarrow p_c$.
\end{proof}

%%%%%%%%%%%%%%%%%%%%%%%%%%%%%%%%%%%%%%%%%%%%%%%%%%%%%%%%%%%%%%%%%%%
%%                                                               %%
%% Use the two commands below for producing your bibliography    %%
%% with bibtex, then comment again the commands and include the  %%
%% content of the .bbl file in this file below the commands.     %%
%%                                                               %%
%%%%%%%%%%%%%%%%%%%%%%%%%%%%%%%%%%%%%%%%%%%%%%%%%%%%%%%%%%%%%%%%%%%

%\bibliographystyle{amsplain}
%\bibliography{yourbibfilename}

\begin{thebibliography}{99}

\bibitem{AP96}
Antal, P., Pisztora, A.: On the chemical distance for supercritical Bernoulli percolation.
\emph{Ann. Probab.} \textbf{24}, 1036--1048 (1996). \MR{1404543}

\bibitem{ADH17}
Auffinger, A., Damron, M., Hanson, J.: \emph{50 years of first-passage
percolation.} University Lecture Series Vol. \textbf{68}, American
Mathematical Society (2017). \MR{3729447}

\bibitem{BD13}
Beffara, V., Duminil-Copin, H.: Planar percolation with a glimpse of Schramm-Loewner evolution. \emph{Probability Surveys} \textbf{10}, 1--50 (2013). \MR{3161674}

\bibitem{Ben10}
Benjamini, I.: Random planar metrics. In \emph{Proceedings of the International Congress of Mathematicians},
vol. IV, pp. 2177--2187. Hindustan Book Agency, New Delhi (2010). \MR{2827966}

\bibitem{Ben13}
Benjamini, I.: Euclidean vs. graph metric. In \emph{Erd\"{o}s Centennial. Bolyai Soc. Math. Stud.} \textbf{25}, 35--57. J\'{a}nos Bolyai Math. Soc., Budapest (2013). \MR{3203593}

\bibitem{BLPR15}
Biskup, M., Louidor, O., Procaccia, E. B., Rosenthal, R.: Isoperimetry in two-dimensional
percolation. \emph{Comm. Pure Appl. Math.} \textbf{68} (9), 1483--1531 (2015). \MR{3378192}

\bibitem{BR06}
Bollob\'{a}s, B., Riordan, O.: \emph{Percolation.} Cambridge University
Press, New York (2006). \MR{2283880}

\bibitem{CN06}
Camia, F., Newman, C.M.:  Critical percolation: the full scaling
limit. \emph{Commun. Math. Phys.} \textbf{268}, 1--38 (2006). \MR{2249794}

\bibitem{CN07}
Camia, F., Newman, C.M.: Critical percolation exploration path and $SLE_6$: a
proof of convergence. \emph{Probab. Theory Relat. Fields} \textbf{139}(3), 473--519 (2007). \MR{2322705}

\bibitem{CD20}
Cerf, R., Dembin, B. Vanishing of the anchored isoperimetric profile in bond
percolation at $p_c$. \emph{Electron. Commun. Probab.} \textbf{25}, no. 2, 7 pp. (2020). \MR{4053905}

\bibitem{Ch70}
Cheeger, J.: A lower bound for the smallest eigenvalue of the Laplacian. In \emph{Proceedings of the Princeton conference in honor of Professor S.
Bochner} 195--199. Princeton Univ. Press, Princeton, N.J. (1970). \MR{0402831}

\bibitem{CK81}
Cox, J.T., Kesten, H.: On the continuity of the time constant of
first-passage percolation. \emph{J. Appl. Probab.} \textbf{18}, 809--819
(1981). \MR{0633228}

\bibitem{Dem20a}
Dembin, B.: Anchored isoperimetric profile of the infinite cluster in supercritical bond percolation is Lipschitz continuous. \emph{Electron. Commun. Probab.} \textbf{25}, no. 34, 13 pp. (2020). \MR{4092764}

\bibitem{Dem20b}
Dembin, B. Existence of the anchored isoperimetric profile in supercritical bond
percolation in dimension two and higher. \emph{ALEA Lat. Am. J. Probab. Math. Stat.} \textbf{17}, 205--252
(2020). \MR{4105293}

\bibitem{DC13}
Duminil-Copin, H.: Limit of the Wulff Crystal when approaching criticality for site percolation on the triangular lattice.
\emph{Electron. Commun. Probab.} \textbf{18}, no. 93, 9 pp. (2013). \MR{3151749}

\bibitem{GMPT17}
Garet, O., Marchand, R., Procaccia, E. B., Th\'{e}ret, M.: Continuity of the
time and isoperimetric constants in supercritical percolation. \emph{Electron. J. Probab.} \textbf{22}, no. 78, 35 pp.
(2017). \MR{3710798}

\bibitem{Gol18a}
Gold, J.: Intrinsic isoperimetry of the giant component of supercritical bond percolation in dimension two.
\emph{Electron. J. Probab.} \textbf{23}, no. 53, 41 pp. (2018). \MR{3814247}

\bibitem{Gol18b}
Gold, J.: Isoperimetry in supercritical bond percolation in dimensions three and
higher. \emph{Ann. Inst. Henri Poincar\'{e} Probab. Stat.} \textbf{54} (4), 2092--2158 (2018). \MR{3865668}

\bibitem{GPS18}
Garban, C., Pete, G., Schramm, O.: The scaling limits of
near-critical and dynamical percolation. \emph{J. Eur. Math. Soc.}
\textbf{20}, 1195--1268 (2018). \MR{3790067}

\bibitem{GM07}
Garet, O., Marchand, R.: Large deviations for the chemical distance
in supercritical Bernoulli percolation. \emph{Ann. Probab.} \textbf{35},
833--866 (2007). \MR{2319709}

\bibitem{Gri99}
Grimmett, G.: \emph{Percolation}, 2nd ed. Springer-Verlag Berlin (1999). \MR{1707339}

\bibitem{Kes86}
Kesten, H.: Aspects of first passage percolation. In \emph{Lecture Notes
in Math.}, Vol 1180, pp. 125--264 Berlin: Springer (1986). \MR{0876084}

\bibitem{Nol08}
Nolin, P.: Near critical percolation in two-dimensions. \emph{Electron. J.
Probab.} \textbf{13}, 1562--1623 (2008). \MR{2438816}

\bibitem{Pet08}
Pete, G.: A note on percolation on $\mathbb{Z}^d$: Isoperimetric profile via exponential cluster
repulsion. \emph{Electron. Commun. Probab.} \textbf{13}, 377--392 (2008). \MR{2415145}

\bibitem{RT10}
Rossignol, R., Th\'{e}ret, M.: Lower large deviations and laws of large numbers
for maximal flows through a box in first passage percolation. \emph{Ann. Inst. Henri
Poincar\'{e} Probab. Stat.}, \textbf{46} (4), 1093--1131 (2010). \MR{2744888}

\bibitem{RT10b}
Rossignol, R., Th\'{e}ret, M.: Law of large numbers for the maximal
flow through tilted cylinders in two-dimensional first passage percolation.
\emph{Stoch. Proc. Appl.} \textbf{120} 873--900 (2010). \MR{2610330}

\bibitem{SW01}
Smirnov, S.,  Werner, W.: Critical exponents for two-dimensional
percolation. \emph{Math. Res. Lett.} \textbf{8}, 729--744 (2001). \MR{1879816}

\bibitem{Tay74}
Taylor, J. E.: Existence and structure of solutions to a class of nonelliptic variational problems.
\emph{Sympos. Math.} \textbf{14}, no. 4, 499--508 (1974). \MR{0420407}

\bibitem{Tay75}
Taylor, J. E.: Unique structure of solutions to a class of nonelliptic variational problems. \emph{Proc.
Sympos. Pure Math.} \textbf{27}, 419--427 (1975). \MR{0388225}

\bibitem{Yao22}
Yao, C.-L.: Convergence of limit shapes for 2D near-critical first-passage percolation, \ARXIV{2104.01211}.  To appear in
\emph{Ann. Inst. Henri Poincar\'{e} Probab. Stat.}
\end{thebibliography}

% add below the content of your .bbl file produced by bibtex.

\begin{acks}
The author was supported by the National Key R\&D Program of China (No. 2020YFA0712700),
the National Natural Science Foundation of China (No. 12288201) and the Key Laboratory of Random Complex Structures and Data Science, CAS (No. 2008DP173182).
\end{acks}

%%%%%%%%%%%%%%%%%%%%%%%%%%%%%%%%%%%%%%%%%%%%%%%%%%%%%%%%%%%%%%%%%%%
%%                                                               %%
%% You have reached the end of your document.                    %%
%%                                                               %%
%%%%%%%%%%%%%%%%%%%%%%%%%%%%%%%%%%%%%%%%%%%%%%%%%%%%%%%%%%%%%%%%%%%

\end{document}